\documentclass[11pt, leqno]{amsart}
\usepackage{graphicx}
\usepackage{amsfonts,delarray,amssymb,amsmath,amsthm,a4,a4wide}
\usepackage{latexsym}
\usepackage{epsfig}
\usepackage{color}
\usepackage[margin=0.98in]{geometry} 
\newtheorem{thm}{Theorem}[section]
\newtheorem{cor}[thm]{Corollary}
\newtheorem{conj}[thm]{Conjecture}
\newtheorem{lem}[thm]{Lemma}
\newtheorem{prop}[thm]{Proposition}
\newtheorem{rem}[thm]{Remark}

\theoremstyle{definition}
\newtheorem{defn}[thm]{Definition}
\numberwithin{equation}{section}
\newcommand{\eb}{\text{\bf{e}}}
\newcommand{\R}{\mathbb R}

\newcommand{\e}{\varepsilon}

\newcommand{\ov}{\overline}
\newcommand{\p}{\partial}
\newcommand{\dist}{\mbox{dist}\,}
\newcommand{\diam}{\mbox{diam}\,}
\newcommand{\trace}{\mbox{trace}\,}

\newcommand{\comment}[1]{}
\def\h{\hspace*{.24in}}

\newenvironment{myindentpar}[1]%
{\begin{list}{}%
         {\setlength{\leftmargin}{#1}}%
         \item[]%
}
{\end{list}}
\begin{document} 

\title[ Monge-Amp\`ere eigenvalue problem on general bounded convex domains]{The eigenvalue problem for the Monge-Amp\`ere operator on general bounded convex domains 
}
\author{Nam Q. Le}
\address{Department of Mathematics, Indiana University, 831 E 3rd St,
Bloomington, IN 47405, USA}
\email{nqle@indiana.edu}
\thanks{The research of the author was supported in part by the National Science Foundation under grant DMS-1500400.}


% ----------------------------------------------------------------
\begin{abstract}
 
In this paper, we study the eigenvalue problem for the Monge-Amp\`ere operator on general bounded convex domains. We prove the existence, uniqueness and variational characterization of the Monge-Amp\`ere eigenvalue.
The convex Monge-Amp\`ere eigenfunctions are shown to be unique up to positive multiplicative constants. Our results are 
the singular counterpart of previous results by P-L. Lions and K. Tso in the smooth, uniformly convex setting. Moreover, we prove the stability of the Monge-Amp\`ere eigenvalue with respect to the 
Hausdorff convergence of the domains. This stability property makes it possible to investigate 
the Brunn-Minkowski,  isoperimetric 
and reverse isoperimetric inequalities for the Monge-Amp\`ere eigenvalue in their full generality. 
We also discuss related existence and regularity results for a class of degenerate Monge-Amp\`ere equations.

\end{abstract}
\maketitle

\section{Introduction and  statement of the main result}
\subsection{The Monge-Amp\`ere eigenvalue problem}
The eigenvalue problem for the Monge-Amp\`ere operator $\det D^2 u$ on smooth, open, bounded and uniformly convex domains $\Omega$ in $\R^n$ ($n\geq 2$) was 
first investigated by Lions \cite{Ls}. He showed that there exist a unique positive constant $\lambda=\lambda(\Omega)$ and a unique (up to positive multiplicative constants) nonzero 
convex function $u\in C^{1,1}(\overline{\Omega})\cap C^{\infty}(\Omega)$ solving the eigenvalue problem for the Monge-Amp\`ere operator $\det D^2 u$: 
\begin{equation}\label{EP}
\left\{
 \begin{alignedat}{2}
   \det D^2 u~& = \lambda|u|^n~&&\text{in} ~  \Omega, \\\
u &= 0~&&\text{on}~ \p\Omega.
 \end{alignedat} 
  \right.
  \end{equation}
Lions \cite{Ls} also gave an interesting stochastic interpretation for $\lambda(\Omega)$. A 
variational characterization of $\lambda$ was found by Tso \cite{Tso} who discovered that for sufficiently smooth, open, bounded, and uniformly convex domains $\Omega$
\begin{equation}
 \lambda(\Omega)=\inf\left\{ \frac{\int_{\Omega} (-u)\det D^2 u~dx}{\int_{\Omega}(-u)^{n+1}~dx}: u\in C^{0,1}(\overline{\Omega})\cap C^2(\Omega),~u~\text{is nonzero, convex in } 
 \Omega,~u=0~\text{on}~\p\Omega\right\}.
 \label{lamT}
\end{equation}
All of the above features of $\lambda$ suggest the well-known properties of the first eigenvalue of the Laplace operator. As such, $\lambda(\Omega)$ is called 
the Monge-Amp\`ere eigenvalue of $\Omega$. The nonzero convex functions $u$ solving (\ref{EP}) are called the Monge-Amp\`ere eigenfunctions. Despite the degeneracy of the right hand side
of (\ref{EP}) near the boundary, the Monge-Amp\`ere eigenfunctions also enjoy the same global smoothness feature as the first eigenfunctions of the Laplace operator.
The fact that $u\in C^{\infty}(\overline{\Omega})$ when $\Omega$ is smooth and uniformly convex was recently proved by the author and Savin \cite{LS}.

A full generalization of the results by Lions and Tso to the $k$-Hessian operator ($1\leq k\leq n$) was carried out by Wang \cite{W}. 
Recall that the Monge-Amp\`ere operator is the $n$-Hessian operator while the Laplace operator is
the $1$-Hessian operator. Wang studied
the first eigenvalue together with its variational characterization for the $k$-Hessian operator. 

Using the variational characterization (\ref{lamT}) of the Monge-Amp\`ere eigenvalue, Salani \cite{Sal} proved a Brunn-Minkowski inequality
for $\lambda(\cdot)$ by showing that $\lambda^{-\frac{1}{2n}}(\cdot)$ (which is positively homogeneous of degree $1$) is concave in the class of 
$C^2_{+}$ domains in $\R^n$ endowed with the Minkowski addition. Here, a domain $\Omega$ in $\R^n$ belongs to the class $C^2_{+}$ if it is an open bounded uniformly convex set
with $C^2$ boundary. More precisely, Salani proved that if $\Omega_0$ and $\Omega_1$ are $C^{2}_{+}$ domains in $\R^n$ 
%and $\alpha\in [0, 1]$ 
then
\begin{equation}
 \lambda(\Omega_{\alpha})^{-\frac{1}{2n}}\geq (1-\alpha) \lambda (\Omega_0)^{-\frac{1}{2n}} + \alpha \lambda(\Omega_1)^{-\frac{1}{2n}}~\text{for all}~\alpha\in[0, 1],
 \label{lamBM}
\end{equation}
where $\Omega_{\alpha}$ denotes the Minkowski linear combination of $\Omega_0$ and $\Omega_1$:
\begin{equation}
\label{Min_lin}
 \Omega_{\alpha}= (1-\alpha)\Omega_0 + \alpha \Omega_1=\left\{(1-\alpha)x_0 + \alpha x_1: x_0\in\Omega_0, x_1\in\Omega_1\right\}.
\end{equation}

Recall that the original form of the Brunn-Minkowski inequality (see, for example, \cite[Theorem 7.1.1]{Sch}) involves volumes of convex bodies, that is, 
compact convex sets with non-empty interior. It states that for convex
bodies $K_0$ and $K_1$ in $\R^n$, we have for all $\alpha\in [0, 1]$
\begin{equation}|K_{\alpha}|^{\frac{1}{n}}\geq (1-\alpha) |K_0|^{\frac{1}{n}} + \alpha |K_1|^{\frac{1}{n}}.
 \label{volBM}
\end{equation}
Note that only convexity is required for $K_0$ and $K_1$ in (\ref{volBM}) and no further regularities on $\p K_0$ and $\p K_1$ are necessary. 

It is thus a natural question to ask if (\ref{lamBM}) also holds for general open bounded convex domains. This question was raised by Salani \cite[p. 85]{Sal}. 
In order to answer this question, we need to develop a theory of the Monge-Amp\`ere eigenvalue for general open bounded convex domains. This is the main goal of the present work. 
We prove the existence, uniqueness and variational characterization of the Monge-Amp\`ere eigenvalue.
The convex Monge-Amp\`ere eigenfunctions are shown to be unique up to positive multiplicative constants. Our results are
the singular counterpart of previous results by  Lions and Tso mentioned above in the smooth, uniformly convex setting. 
Moreover, we prove the stability of the Monge-Amp\`ere eigenvalue with respect to the 
Hausdorff convergence of the domains.  

Unless otherwise indicated, bounded open convex domains $\Omega\subset\R^n$ are throughout assumed to 
have non-empty interior. Our main result states as follows.

\begin{thm} [The Monge-Amp\`ere eigenvalue problem]
\label{lam_thm}
 Let $\Omega$ be a bounded open convex domain in $\R^n$. Define the constant $\lambda=\lambda[\Omega]$ via the variational formula:
\begin{equation}
 \lambda[\Omega] =\inf\left\{ \frac{\int_{\Omega} (-u)\det D^2 u~ dx }{\int_{\Omega}(-u)^{n+1}~ dx}: u\in C(\overline{\Omega}),
 ~u~\text{is convex, nonzero in } \Omega,~ u=0~\text{on}~\p\Omega\right\}.
\label{lam_def}
 \end{equation}
 Then, 
 \begin{myindentpar}{1cm}
 (i) We have the estimates \begin{equation}c(n) |\Omega|^{-2}\leq \lambda[\Omega]\leq C(n) |\Omega|^{-2}
  \label{lam_est}
 \end{equation}
 for some geometric constants $c(n), C(n)$ depending only on the dimension $n$.\\
 (ii) There exists a 
 nonzero convex solution $u\in C^{0,\beta}(\overline{\Omega})\cap C^{\infty}(\Omega)$ for all $\beta\in (0, 1)$ to the eigenvalue problem
 \begin{equation}
 \left\{
 \begin{alignedat}{2}
   \det D^{2} u~&=\lambda |u|^{n} \h~&&\text{in} ~\Omega, \\\
u &=0\h~&&\text{on}~\p \Omega.
 \end{alignedat}
 \right.
 \label{EVP_eq}
\end{equation}
Thus the infimum in (\ref{lam_def}) is achieved.\\
(iii) The eigenvalue-eigenfunction pair $(\lambda, u)$ to (\ref{EVP_eq}) is unique in the following sense: If the pair $(\Lambda, v)$ 
satisfies $\det D^2 v =\Lambda |v|^n$ in $\Omega$ where $\Lambda>0$ is a positive constant and 
$v\in C(\overline{\Omega})$ is convex, nonzero
with $v=0$ on $\p\Omega$, then $\Lambda=\lambda$ and $v=mu$ for some positive constant $m$.\\
(iv) $\lambda[\cdot]$ is stable with respect to the Hausdorff convergence of the domains: If $\{\Omega_m\}_{m=1}^{\infty}\subset\R^n$ is a sequence of open bounded convex domains that 
converges in the Hausdorff distance to an open bounded convex domain $\Omega$, then $\lim_{m\rightarrow \infty}\lambda[\Omega_m]=\lambda[\Omega].$
\end{myindentpar}
\end{thm}
The proof of Theorem \ref{lam_thm} will be given in Section \ref{sec_proof}. In Section \ref{sec_com}, we will comment on its proof.\\
When $u$ is merely a convex function, $\det D^2 u~ dx$ is interpreted as the Monge-Amp\`ere measure associated with $u$. Basic facts concerning the Monge-Amp\`ere measure
and the Monge-Amp\`ere equation will be recalled in Section \ref{MA_sec}.
\begin{rem}~
\begin{myindentpar}{1cm}
 (i) Once Theorem \ref{lam_thm} is proved, we find that for smooth, open, bounded and uniformly convex domains $\Omega$ in $\R^n,$ 
 $$\lambda (\Omega)=\lambda[\Omega].$$
 (ii) We call $\lambda[\Omega]$ defined in Theorem \ref{lam_thm} the Monge-Amp\`ere eigenvalue of the 
 domain $\Omega$, and nonzero convex functions $u\in C(\overline{\Omega})$ solving (\ref{EVP_eq}) the Monge-Amp\`ere eigenfunctions of the 
 domain $\Omega$.
 \end{myindentpar}
\end{rem}
As it turns out, the theory of the Monge-Amp\`ere eigenvalue on general open bounded convex domains together with its stability property in Theorem \ref{lam_thm} allows us
to study 
the maximum and minimum problem
for the Monge-Amp\`ere eigenvalue of open bounded convex sets under a volume constraint.
In this extremal problem, the extremal sets are shown to exist but could potentially be non-strictly convex. We will discuss this problem in the next section.
\subsection{The Brunn-Minkowski, isoperimetric and reverse isoperimetric inequalities}
Using the stability with respect to the Hausdorff convergence of the Monge-Amp\`ere eigenvalue, Salani's Brunn-Minkowski inequality (\ref{lamBM}) immediately extends
to bounded convex domains.
\begin{thm}[The Brunn-Minkowski inequality]
\label{BMcor}
If $\Omega_0$ and $\Omega_1$ are open bounded convex domains in $\R^n$ and $\alpha\in [0, 1]$, then 
$$\lambda\left[(1-\alpha)\Omega_0 + \alpha\Omega_1\right]^{-\frac{1}{2n}}\geq (1-\alpha) \lambda [\Omega_0]^{-\frac{1}{2n}} + \alpha \lambda[\Omega_1]^{-\frac{1}{2n}}.$$
\end{thm}
We will give a proof of  Theorem \ref{BMcor} in Section \ref{iso_sec}.

In view of the estimates (\ref{lam_est}), a very natural problem concerning the Monge-Amp\`ere eigenvalue is to determine, among open bounded convex sets having a fixed volume, those
with 
the largest or smallest Monge-Amp\`ere eigenvalue. 
The following theorem is concerned with the maximum and minimum of the Monge-Amp\`ere eigenvalue. We call a set $K$  centrally symmetric if $-x\in K$
whenever $x\in K$.
\begin{thm}
(i) \emph{(Maximum of the Monge-Amp\`ere eigenvalue)} Among all bounded open convex sets in $\R^n$ having a fixed positive volume, the ball has the largest Monge-Amp\`ere 
eigenvalue. That is, if $\Omega$ is a bounded open convex domain and let 
$\mathcal{B}(\Omega)$ be a ball having the same volume as $\Omega$, then 
\begin{equation}\label{iso_ineq}\lambda[\Omega]\leq \lambda[\mathcal{B}(\Omega)]=
\frac{|B_1|^2 \lambda (B_1)}{|\Omega|^2}.
\end{equation}
(ii) \emph{(Minimum of the Monge-Amp\`ere eigenvalue)} Among all bounded open convex sets $\Omega$ in $\R^n$ having a fixed positive volume, there exists a convex $S$ with the smallest Monge-Amp\`ere eigenvalue:
\begin{equation}
\label{rev_ineq} \lambda[\Omega]\geq \lambda[S]~\text{for all convex}~\Omega ~\text{with}~|\Omega|=|S|.
\end{equation}
(iii) \emph{(Minimum of the Monge-Amp\`ere eigenvalue among centrally symmetric convex sets)} Among all centrally symmetric bounded open convex sets $\Omega$ in $\R^n$ having a fixed positive volume, there exists a 
centrally symmetric convex $K$ with the smallest Monge-Amp\`ere eigenvalue:
\begin{equation}
\label{rev_ineq_sym} \lambda[\Omega]\geq \lambda[K]~\text{for all centrally symmetric convex}~\Omega ~\text{with}~|\Omega|=|K|.
\end{equation}
\label{extreme_thm}
\end{thm}
The proof of Theorem \ref{extreme_thm} will be given in Section \ref{iso_sec}.

When the domain $\Omega$ is smooth and uniformly convex, Theorem \ref{extreme_thm} (i) is due to Brandolini-Nitsch-Trombetti \cite[Theorem 3.2]{BNT}.  

Following the terminology in \cite{BNT}, we call (\ref{iso_ineq}) the 
{\it isoperimetric inequality for the Monge-Amp\`ere eigenvalue}. It is then reasonable to call (\ref{rev_ineq}) and (\ref{rev_ineq_sym}) the 
{\it reverse isoperimetric inequalities for the Monge-Amp\`ere eigenvalue.}  In choosing this terminology, we are partially motivated by Ball's reverse isoperimetric inequality for the classical area functional \cite{B}
which states that modulo affine transformations, among all convex bodies in $\R^n$, the $n$-dimensional regular simplex has largest surface area for a given volume, while among centrally symmetric convex bodies, the $n$-dimensional cube is extremal.

It would be interesting to determine the extremal convex sets in Theorem \ref{extreme_thm} (ii) and (iii). In 
view of the reverse isoperimetric inequality due to Ball \cite{B}, we have the following conjecture on the possible candidate for the minimum of the Monge-Amp\`ere eigenvalue.
\begin{conj}
\label{lam_conj}
\begin{myindentpar}{1cm}~\\
(i) Among all bounded open convex sets in $\R^n$ having a fixed positive volume, the n-dimensional regular simplex (that is the interior of 
the convex hull of $(n+1)$ equally spaced points in $\R^n$) has the smallest Monge-Amp\`ere eigenvalue.\\
(ii) Among all open bounded  centrally symmetric convex sets in $\R^n$ having a fixed positive volume, the n-dimensional cube has the smallest Monge-Amp\`ere eigenvalue.
\end{myindentpar}
\end{conj}

By the affine invariant property of the Monge-Amp\`ere eigenvalue (see Proposition \ref{lam_nice}), the images of extremal sets in Theorem \ref{extreme_thm} and 
Conjecture \ref{lam_conj} under volume preserving affine transformations are also extremal sets. 

For results related to Theorem \ref{extreme_thm} (ii) and (iii) 
and Conjecture \ref{lam_conj} for the case of the Laplacian eigenvalue, see \cite{BF} and the references therein. In this case, the complete solution is only available 
in dimensions $n=2$ for axisymmetric extremal sets; see Proposition 10 and Theorem 12 in \cite{BF}.
\subsection{On the proof of Theorem \ref{lam_thm} and related results}
\label{sec_com}
Let us say a few words about the proof of Theorem \ref{lam_thm} to be presented in Section \ref{sec_proof}. The bounds on $\lambda$ in (\ref{lam_est}) are based on standard techniques in the Monge-Amp\`ere equation.
They will be proved in Lemma \ref{Alem} and Corollary \ref{eigen_est}. 

Next, we describe how to solve (\ref{EVP_eq}) for $\lambda=\lambda[\Omega]$ defined by (\ref{lam_def}).
We construct, as in Tso \cite{Tso}, a nonzero convex solution $u\in C(\overline{\Omega})$ of (\ref{EVP_eq}) as a limit, when $p\nearrow n$ of nonzero convex solutions to the degenerate Monge-Amp\`ere equations
\begin{equation}
 \left\{
 \begin{alignedat}{2}
   \det D^{2} u~&=\lambda |u|^{p} \h~&&\text{in} ~\Omega, \\\
u &=0\h~&&\text{on}~\p \Omega.
 \end{alignedat}
 \right.
 \label{DMA0}
\end{equation}
This class of equations was studied by Tso \cite{Tso} for smooth, bounded, and uniformly convex domains and Hartenstine \cite{Har2} for bounded and strictly convex domains. It turns out 
that we can obtain
the same qualitative results for the above equations as in \cite{Tso, Har2} for general bounded convex domains; see Theorem \ref{sub_thm}. It is interesting to note that the 
above equations with $\Omega$ being a polytope also arise naturally
in Donaldson's analysis of the Abreu equation \cite{D2} in complex geometry. We will discuss all these in Section \ref{D_sec}.

The proof of the global almost Lipschitz property of the Monge-Amp\`ere eigenfunctions in Proposition \ref{alm_Lip}, that is, 
$u\in C^{0,\beta}(\overline{\Omega})$ for all $\beta\in (0, 1)$, is based on an iteration argument using the maximum principle. 
Interestingly, our argument also shows that  when $\Omega$ is a general open bounded convex domain, nonzero convex solutions to (\ref{DMA0}) are globally almost Lipschitz for all $n-2\leq p\neq n$ (see Proposition \ref{pn2_prop}) but it does not indicate whether the same conclusion holds for $p\in [0, n-2)$ when $n\geq 3$.
On the other hand, when $\Omega$ is an open bounded domain with uniformly convex smooth boundary, nonzero convex solutions to (\ref{DMA0})  are globally smooth as the Monge-Amp\`ere eigenfunctions exactly when $p$ is a positive integer, regardless whether it is less than $n-2$ or not. 
For reader's convenience, we state here an optimal global regularity result which complements Theorem 1.4 in \cite{LS} on the global smoothness of the Monge-Amp\`ere eigenfunctions. 
\begin{thm} [see Theorem \ref{Rlem}] Let $\Omega$ be an open bounded domain with uniformly convex smooth boundary. Let $p>0$ and $p\neq n$. 
 Assume that $u\in C(\overline{\Omega})$ is a nonzero convex solution of 
 \begin{equation*}
 \left\{
 \begin{alignedat}{2}
   \det D^{2} u~&= |u|^p \h~&&\text{in} ~\Omega, \\\
v &=0\h~&&\text{on}~\p \Omega.
 \end{alignedat}
 \right.
\end{equation*}
\begin{myindentpar}{1cm}
 (i) If $p$ is a positive integer then $u\in C^{\infty}(\overline{\Omega}).$\\
 (ii) If  $p$ is a not an integer then $u\in C^{2 + [p],\beta}(\overline{\Omega}) $ for some $\beta=\beta (n, p)\in (0, 1)$ but $u\not\in C^{3+[p]}(\overline{\Omega}).$ Here,
 $[p]$ denotes the greatest integer not exceeding $p$.
\end{myindentpar}
\label{Rlem0}
\end{thm}

Another way to construct nonzero convex solutions to (\ref{EVP_eq}) with $\lambda$ replaced by a (possible different) positive constant $\Lambda$ is by approximation.
We approximate $\Omega$ by a sequence of open, bounded, smooth and uniformly convex domains $\{\Omega_m\}$ in $\R^n$ 
 that converges to $\Omega$ in the Hausdorff distance. For each $m$, we
take $u_m\in C^{1,1}(\overline{\Omega_m})\cap C^{\infty}(\Omega_m)$ with $\|u_m\|_{L^{\infty}(\Omega_m)}=1$ to be 
a convex eigenfunction of the Monge-Amp\`ere operator on $\Omega_m$:
\begin{equation*}
\left\{
 \begin{alignedat}{2}
   \det D^2 u_m~& = \lambda(\Omega_m)|u_m|^n~&&\text{in} ~  \Omega_m, \\\
u_m &= 0~&&\text{on}~ \p\Omega_m.
 \end{alignedat} 
  \right.
  \end{equation*}
We then let $m\rightarrow\infty$ to obtain a nonzero convex solution to (\ref{EVP_eq}) with $\lambda$ replaced by $\Lambda=\lim_{m\rightarrow \infty}\lambda(\Omega_m)$; see 
Proposition
\ref{Lam_prop}.

Thus, there arises naturally the question of the uniqueness of the eigenvalue and eigenfunctions for the Monge-Amp\`ere operator on general open bounded convex domains.
{\it This uniqueness issue is the main difficulty in the proof of Theorem \ref{lam_thm}.} The proof of uniqueness in \cite{Ls} (see also \cite{W}) uses crucially
the Lipschitz property of the eigenfunctions. This is possible when $\Omega$ is smooth and uniformly convex. In our situation, the eigenfunctions are not
Lipschitz in general and we need new arguments. The proof of the uniqueness of the  Monge-Amp\`ere eigenvalue in {\it Step 1} of Proposition \ref{uni_lam}  is a simple application of  the key estimate in Proposition \ref{detC} 
which is a form of nonlinear integration by parts. 
\begin{prop}[Nonlinear integration by parts]
 \label{detC} Let $\Omega$ be a bounded open convex domain in $\R^n$.
 Suppose that $u, v\in C(\overline{\Omega})\cap C^5 (\Omega)$ are strictly convex functions in $\Omega$ with $u=v=0$ on $\p\Omega$ and that there is a constant $M>0$ such that
 \begin{equation}\int_{\Omega}(\det D^2 u)^{\frac{1}{n}}  (\det D^2 v)^{\frac{n-1}{n}}~dx\leq M,~\text{and}~\int_{\Omega}\det D^2 v~dx\leq M.
 \label{detuv}
 \end{equation} Then
\begin{equation} \label{term0}\int_{\Omega} |u|\det D^2 v~dx \geq \int_{\Omega} |v|(\det D^2 u)^{\frac{1}{n}} (\det D^2 v)^{\frac{n-1}{n}}~dx.
\end{equation}
\end{prop}
We will prove Proposition \ref{detC} in Section \ref{detC_sec}.

The proof of the uniqueness of the  Monge-Amp\`ere eigenfunctions up to positive multiplicative constants in {\it Step 2} of Proposition \ref{uni_lam} relies on a delicate application of Proposition \ref{detC} in its full strength.
This application uses crucially the global almost Lipschitz property of the eigenfunctions in Proposition \ref{alm_Lip} which gives, among other things, an important integrability property of the eigenfunctions and their 
Hessians as stated in the following lemma.
\begin{lem} [see Lemma \ref{inte_lem}]
\label{inte_lem0} Let $\Omega$ be a bounded open convex domain in $\R^n$.
Suppose that $u, v\in C(\overline{\Omega})$ are nonzero convex eigenfunctions to (\ref{EVP_eq}). 
Then, there exists a constant $C=C(n,\Omega)$ depending only $n$ and $\Omega$ so that 
$$\int_{\Omega} \Delta u |v|^{n-1}~dx\leq C \|u\|_{L^{\infty}(\Omega)} \|v\|^{n-1}_{L^{\infty}(\Omega)}.$$
\end{lem}
Finally, the stability of the Monge-Amp\`ere eigenvalue with respect to the Hausdorff convergence of the domains follows from the approximability and uniqueness of the Monge-Amp\`ere eigenvalue.
\vglue 0.2cm
Throughout the paper, the dimension $n$ ($\geq 2$) of $\R^n$ is fixed. The open ball centered at the origin with radius $r$ is denoted by $B_r$.
We use $c(n), C(n)$ to denote positive constants depending only on $n$. We think of $c$ and $C$ as being small and large constants, respectively.  We denote the volume (or Lebesgue
measure) and diameter of a set $K$ in $\R^n$ by $|K|$ and $\diam K$, respectively. The distance function to a closed set $\Gamma$ is denoted by $\dist(\cdot,\Gamma)$.
The dependence of various constants also
on other parameters like $p,$ $\diam\Omega$, $|\Omega|$ will be denoted by $c(p, \diam\Omega, |\Omega|)$, $C(p, \diam\Omega, |\Omega|)$, etc.
\vglue 0.2cm
The rest of the paper is organized as follows. In Section \ref{MA_sec}, we recall basic facts on the Monge-Amp\`ere equation. In Section \ref{est_sec}, we present various estimates related to the Monge-Amp\`ere operators %leading to sharp bounds for the Monge-Amp\`ere eigenvalue and solutions to 
and
degenerate Monge-Amp\`ere equations. Then we give the proof of Proposition \ref{detC}. In Section \ref{D_sec}, we study solutions to the degenerate Monge-Amp\`ere equations and their minimality with respect to certain Monge-Amp\`ere functional. We prove Theorem \ref{lam_thm} in Section \ref{sec_proof}.  Theorems \ref{lamBM} and \ref{extreme_thm} will be proved in Section \ref{iso_sec}. In Section \ref{gl_sec}, we prove Theorem \ref{Rlem0} restated as Theorem \ref{Rlem}.
\section{Preliminaries on the Monge-Amp\`ere equation}
\label{MA_sec}
In this section, we recall basic facts on the Monge-Amp\`ere equation. We refer the reader to the books by Guti\'errez \cite{G} and Figalli \cite{Fi} for more details.

Let $\Omega$ be an open convex domain of $\R^{n}$. We define
 the subdifferential of a convex function \(u:\Omega \to \R\) at $x\in\Omega$ by
 $$
\partial u (x):=\{p\in \R^{n}\,:\, u(y)\ge u(x)+p\cdot (y-x)\quad \forall\, y \in \Omega\}.
$$
\begin{defn} [The Monge-Amp\`ere measure]
\label{MAdef}
Let $u:\Omega\rightarrow \R$ be a convex function.
Given $E\subset \Omega$, we define
$$\p u(E) = \bigcup_{x\in E} \p u(x).$$
Let
$$Mu(E) = |\p u(E)|.$$
Then (see \cite[Theorem 1.1.13]{G}), $Mu: \mathcal{S}\rightarrow R\cup\{\infty\}$ is a measure, finite on compact sets where $\mathcal{S}$ is the Borel $\sigma$-algebra defined by
$$\mathcal{S}=\{E\subset \Omega: \p u(E) ~\text{is Lebesgue measurable}\}.$$
$Mu$ is called the Monge-Amp\`ere measure associated with the convex function $u$.
\end{defn}

If \(u\in C^{2}(\Omega)\), then we can show using a change of variables that 
$
Mu=\det D^{2} u(x)\,dx$ in $\Omega$.

\begin{defn}[Aleksandrov solutions]\label{ch1:def:alesol}Given an open convex set \(\Omega\) and a Borel measure \(\mu\) on \(\Omega\), a convex
function \(u:\Omega \to \R\) is called an \emph{Aleksandrov solution} to the Monge-Amp\`ere equation
$$
\det D^{2} u =\mu,
$$
if $\mu=Mu$ as Borel measures.
\end{defn}

When \(\mu=f\,dx\) we will simply say that \(u\) solves 
\begin{equation*}
\det D^2 u=f.
\end{equation*}

In what follows, by a slight abuse of notation, we will use $\det D^2 u$ to denote the Monge-Amp\`ere measure $Mu$ for a general convex function $u$. Thus, for all Borel set 
$E\subset\Omega$,
$$\int_{E}\det D^2u~ dx = Mu(E)= |\p u(E)|$$
and
$$\int_E |u|\det D^2 u~dx=\int_E |u| dMu.$$

The basic existence and uniqueness result for solutions to the Dirichlet problem with zero boundary data is given in the following (see \cite[Theorem 1.6.2]{G},
\cite[Theorem 1]{Har1} and \cite[Theorem 2.13]{Fi}):
\begin{thm}[The Dirichlet problem] 
Let $\Omega$ be a bounded open convex domain in $\R^n$, and let $\mu$ be a nonnegative Borel measure in $\Omega$.
Then there exists a unique convex function $u\in C(\overline{\Omega})$ that is an Aleksandrov solution of
\begin{equation}
\label{MA_Dircase}
 \left\{
 \begin{alignedat}{2}
   \det D^{2} u~&=\mu \h~&&\text{in} ~\Omega, \\\
u &=0\h~&&\text{on}~\p \Omega.
 \end{alignedat}
 \right.
\end{equation}
\label{exi_thm}
\end{thm}

We note that in \cite[Theorem 1.6.2]{G}, the domain $\Omega$ in Theorem \ref{exi_thm} was required to be strictly convex. However, the strict convexity of $\Omega$ can be removed
for the Dirichlet problem with zero boundary data as shown in \cite[Theorem 1]{Har1} and \cite[Theorem 2.13]{Fi}.

The following is the celebrated Aleksandrov's maximum principle (see \cite[Theorem 2.8]{Fi} and \cite[Theorem 1.4.2]{G}):

\begin{thm}[Aleksandrov's maximum principle]\label{Alex_thm} Let  $\Omega\subset\R^n$ be an open, bounded and convex domain. Let $u\in C(\overline{\Omega})$ be a convex function.
If $u=0$ on $\p \Omega$, then
$$
|u(x)|^{n}\le C(n)(\emph{diam}\Omega)^{n-1}\emph{dist}(x,\partial \Omega)\int_{\Omega}\det D^2 u~dx\qquad \text{ for all } x\in\Omega.
$$
 \end{thm}

We will also use the following comparison principle; see \cite[Theorem 2.10]{Fi} and \cite[Theorem 1.4.6]{G}.

\begin{lem}[Comparison principle]
\label{comp_prin}
 Let  $\Omega\subset\R^n$ be an open, bounded and convex domain. Let $u, v\in C(\overline{\Omega})$ be convex functions.
If $u\ge v$ on $\p \Omega$ and in the sense of Monge-Amp\`ere measures
$$
\det D^{2}u \leq \det D^{2}v\quad \text{in}~\Omega, 
$$
then $u\ge v$ in $\Omega$.
\end{lem}

We have the following compactness of solutions to the Monge-Amp\`ere equation; see \cite[Corollary 2.12]{Fi} and \cite[Lemma 5.3.1]{G}.

\begin{thm}[Compactness of solutions to the Monge-Amp\`ere equation]
\label{compact_thm}
Let $\{\Omega_k\}_{k=1}^{\infty}\subset \R^n$ be a sequence of open bounded convex domains that converges to an open bounded convex domain $\Omega$
in the Hausdorff distance. 
Let $\{\mu_k\}_{k=1}^{\infty}$ be a sequence of 
nonnegative Borel measures with $\sup_k\mu_k(\Omega_k)<\infty$ and which
converges weakly$^\ast$ to a Borel measure $\mu$.
For each $k$, let $u_k\in C(\overline{\Omega_k}) $ be the convex Aleksandrov solution of
\begin{equation*}
 \left\{
 \begin{alignedat}{2}
   \det D^{2} u_k~&=\mu_k \h~&&\text{in} ~\Omega_k, \\\
u_k &=0\h~&&\text{on}~\p \Omega_k.
 \end{alignedat}
 \right.
\end{equation*}
Then $u_k$ converges locally uniformly in $\Omega$ to the convex Aleksandrov solution of
\begin{equation*}
 \left\{
 \begin{alignedat}{2}
   \det D^{2} u~&=\mu \h~&&\text{in} ~\Omega, \\\
u &=0\h~&&\text{on}~\p \Omega.
 \end{alignedat}
 \right.
\end{equation*}
\end{thm}

We will use the following fact, which is a consequence of John's lemma (see also \cite[Lemma A.13]{Fi}, \cite[Theorem 1.8.2]{G} and \cite[Theorem 10.12.2]{Sch}).
\begin{lem}[John's lemma]
 \label{J_lem}
 Let $\Omega$ be an open, bounded and convex set in $\R^n$ with nonempty interior. Then there exists an affine transformation $T:\R^n\rightarrow \R^n$ with $\det T=1$ (such 
 a transformation is called unimodular) such that
 \begin{equation*}B_{R}\subset T(\Omega)\subset B_{nR}~\text{for some~}R>0.
 \end{equation*}
\end{lem}

We call an open convex set $\Omega$ {\it normalized} if it satisfies $B_{R}\subset \Omega\subset B_{nR}$ for some $R>0$.

Finally, we will use the following simple interior regularity result whose proof we include for reader's convenience.
\begin{prop}
\label{reg_prop}
  Let $0\leq p<\infty$ and let $\Omega$ be an open, bounded convex set in $\R^n$ with nonempty interior.
  Assume that $u\in C(\overline{\Omega})$ is a nonzero convex Aleksandrov solution  to the Dirichlet problem
\begin{equation*}
 \left\{
 \begin{alignedat}{2}
   \det D^{2} u~&=|u|^{p} \h~&&\text{in} ~\Omega, \\\
u &=0\h~&&\text{on}~\p \Omega.
 \end{alignedat}
 \right.
\end{equation*}
Then $u$ is strictly convex in $\Omega$ and $u\in C^{\infty}(\Omega)$.
\end{prop}
We recall that a convex function $u$ on an open bounded convex domain $\Omega$ is said to be strictly convex in $\Omega$, if for any $x\in\Omega$ and $p\in\p u(x)$, 
$$u(z)> u(x) + p\cdot (z-x)~\text{for all~} z\in\Omega\backslash \{x\},$$
that is, any supporting hyperplane to $u$ touches its graph at only one point.
\begin{proof} Let $M=\|u\|_{L^{\infty}(\Omega)}$. For each $\e\in (0, M)$, let $\Omega(\e)=\{x\in\Omega: u(x)\leq -\e\}$. Since $u\in C(\overline{\Omega})$ is convex, the set $\Omega(\e)$ is convex
with nonempty interior. Let us denote $\Omega'=\Omega(\e)$ for brevity. 

Since $\e^p\leq \det D^2 u= |u|^p\leq M^p$ in $\Omega'$ and $u=-\e$ on $\p\Omega'$, the function $u$ is strictly convex in $\Omega'$ by the
localization theorem of Caffarelli \cite{C1} (see also \cite[Corollary 5.2.2]{G} and \cite[Theorem 4.10]{Fi}). Thus, by Caffarelli's $C^{1,\alpha}$ estimates \cite{C3}
(see also \cite[Theorem 5.4.8]{G} and \cite[Theorem 4.20]{Fi}), 
$u\in C^{1,\alpha}_{loc}(\Omega')$ for some $\alpha\in (0, 1)$ depending only on $n, p, \e$ and $M$. Now, using Caffarelli's $C^{2,\alpha}$ estimates \cite{C2}, we have
$u\in C^{2,\alpha}_{loc}(\Omega')$. In the interior of $\Omega'$, the equation $\det D^2 u= |u|^p$ now becomes uniformly elliptic with $C^{2,\alpha}$ right hand side.
By a simple bootstrap argument, we
have $u\in C^{\infty}_{loc}(\Omega')$. Since $\e\in (0, M)$ is arbitrary, we conclude $u\in C^{\infty}(\Omega)$ and $u$ is strictly convex in $\Omega$.
\end{proof}

\section{Estimates}
\label{est_sec}
\subsection{Estimates for the Monge-Amp\`ere eigenvalue}
In this subsection, we establish various estimates that will give the optimal bounds (up to a bounded constant depending only on the dimension) 
from below and above for $\lambda[\cdot]$ defined by (\ref{lam_def}) on general open bounded convex domains. These estimates equally apply to 
the Monge-Amp\`ere eigenvalue $\lambda(\cdot)$ of open, bounded, smooth, and uniformly convex domains as defined in (\ref{EP}).
\begin{lem} 
\label{Alem}
Let $\Omega$ be an open, bounded convex domain in $\R^n$. Let $p\geq 0$.\\
(i)  Let $u\in C(\overline{\Omega})$ be a convex function in $\Omega$ with $u=0$ on $\p\Omega$, and $\|u\|_{L^{\infty}(\Omega)}=1$. Then 
$$c(n,p) |\Omega| \leq \int_{\Omega} |u|^{p}~dx\leq
|\Omega|$$
and
$$\int_{\Omega} (-u) \det D^2 u~dx\geq c(n) |\Omega|^{-1}.$$
(ii) There exists a convex function $u\in C(\overline{\Omega})\cap C^{\infty}(\Omega)$ with $u=0$ on $\p\Omega$, and $\|u\|_{L^{\infty}(\Omega)}=1$ such that 
$$ \int_{\Omega} (-u) \det D^2 u~dx
\leq C(n) |\Omega|^{-1}.$$
If $\Omega$ is open, bounded, smooth, and uniformly convex then we can choose $u$ to satisfy additionally $u\in C^{0, 1}(\overline{\Omega})$.\\
(iii) Let $M>0$ be a positive constant. Assume that $u\in C(\overline{\Omega})$ is a nonzero convex Aleksandrov solution  to the Dirichlet problem
\begin{equation*}
 \left\{
 \begin{alignedat}{2}
   \det D^{2} u~&=M|u|^{p} \h~&&\text{in} ~\Omega, \\\
u &=0\h~&&\text{on}~\p \Omega.
 \end{alignedat}
 \right.
\end{equation*}
Then 
\begin{equation*}
 c(n, p)|\Omega|^{-2}\leq M\|u\|^{p-n}_{L^{\infty}(\Omega)}\leq C(n, p)|\Omega|^{-2}.
\end{equation*}

\end{lem}
\begin{proof}
Under the unimodular affine transformations $T:\R^n\rightarrow\R^n$ with $\det T=1$:
 $$\Omega\rightarrow T(\Omega),~u(x)\rightarrow u(T^{-1}x),$$
 the equation $\det D^2 u =M|u|^p$, the quantities 
 $$ \int_{\Omega} |u|^{p}~dx, \int_{\Omega} (-u) \det D^2 u~dx, \|u\|_{L^{\infty}(\Omega)}~ \text{and}~ |\Omega|$$
 are unchanged. Thus, by John's lemma, Lemma \ref{J_lem},
 we can assume that $\Omega$ is normalized, that is 
 $$B_R\subset \Omega\subset B_{nR}~\text{for some}~R>0.$$
(i) Let $u\in C(\overline{\Omega})$ be a convex function in $\Omega$ with $u=0$ on $\p\Omega$, and $\|u\|_{L^{\infty}(\Omega)}=1$. Then clearly, 
$$\int_{\Omega} |u|^{p}~dx\leq |\Omega|.$$
 Assume that $u(x_0)=-1$ where
$x_0\in\Omega$. Then for any $x\in\Omega$, the ray $x_0 x$ intersects $\p\Omega$ at $z$. We have $x=\alpha x_0 + (1-\alpha) z$ for some $\alpha\in[0,1]$. By the convexity of $u$
and $u=0$ on $\p\Omega$, we have
$$u(x) \leq \alpha u(x_0) + (1-\alpha) u(z) = -\alpha.$$
It follows from $\Omega\subset B_{nR}$ that for all $x\in\Omega$, we have
$$|u(x)|\geq\alpha =\frac{|x-z|}{|x_0-z|} \geq \frac{\dist(x,\p\Omega)}{\diam \Omega}\geq c(n) R^{-1} \dist(x,\p\Omega).$$
Then, using $B_R\subset\Omega$ and $|\Omega|\leq C(n)R^n$, we get 
\begin{eqnarray}\int_{B_{R/2}} |u|^{p}~dx&\geq& 
c(n,p) R^{-p}\int_{B_{R/2}} \dist^{p}(x,\p\Omega)~ dx \nonumber\\ &\geq& c(n,p)R^{-p} \int_{B_{R/2}} \dist^{p}(x, \p B_R) ~dx\geq c(n,p) R^n\geq c(n,p) |\Omega|.
\label{int_half}
\end{eqnarray}
Thus
$$\int_{\Omega}|u|^p~dx\geq \int_{B_{R/2}} |u|^{p}~dx \geq c(n,p) |\Omega|.$$

Next, let $$\Omega'=\{x\in\Omega: u(x)\leq -1/2\}.$$ Applying the
 Aleksandrov maximum principle, Theorem \ref{Alex_thm}, to $u+\frac{1}{2}$ on $\Omega'$, and noting that $\|u+\frac{1}{2}\|_{L^{\infty}(\Omega^{'})}=\frac{1}{2}$, we find
 $$\frac{1}{2} \leq  C(n) \diam \Omega'\left(\int_{\Omega'}\det D^2 u~dx\right)^{1/n} \leq C(n) R\left(\int_{\Omega'}\det D^2 u~dx\right)^{1/n} .$$
 Hence 
 $$\int_{\Omega} (-u)\det D^2 u~dx \geq \int_{\Omega'} (-u)\det D^2 u~dx \geq \frac{1}{2}\int_{\Omega'}\det D^2 u~dx \geq c(n)R^{-n}\geq c(n) |\Omega|^{-1}. $$
(ii) Let $v\in C(\overline{\Omega})$ be the convex Aleksandrov solution to 
\begin{equation*}
 \left\{
 \begin{alignedat}{2}
   \det D^{2} v~&=\frac{1}{R^{2n}} \h~&&\text{in} ~\Omega, \\\
v &=0\h~&&\text{on}~\p \Omega.
 \end{alignedat}
 \right.
\end{equation*}
This solution exists and is unique by Theorem \ref{exi_thm}. Clearly, $v\in C^{\infty}(\Omega)$. 
If $\Omega$ is smooth, open, bounded and uniformly convex then $v\in C^{0, 1}(\overline{\Omega})$. We can see this by constructing an explicit barrier or by recalling
the classical result of Caffarelli-Nirenberg-Spruck \cite[Theorem 1]{CNS}.

We show that 
\begin{equation}
\label{vunit}
 \|v\|_{L^{\infty}(\Omega)}\geq \frac{1}{2}.
\end{equation}
Indeed, let 
$$w(x) = \frac{1}{2R^2}(|x|^2-R^2).$$
Then $\det D^2 w = \frac{1}{R^{2n}}=\det D^2 v~\text{in}~\Omega$
and, since $\Omega\supset B_R$, 
$w\geq 0=v~\text{on}~\p\Omega.$
By the comparison principle, Lemma \ref{comp_prin}, we find that $v\leq w$ in $\Omega$. In particular, $v(0)\leq w(0)=-\frac{1}{2}$. It follows that
$$\|v\|_{L^{\infty}(\Omega)}=-\min_{\overline{\Omega}} v\geq -v(0) \geq \frac{1}{2}.$$
Let $u=\alpha v$ where $\alpha =\frac{1}{\|v\|_{L^{\infty}(\Omega)}}\leq 2,$ by (\ref{vunit}). Then $u\in C(\overline{\Omega})$ is convex, $u=0$ on $\p\Omega$ with $\|u\|_{L^{\infty}(\Omega)}=1$ and we have
$$ \int_{\Omega}(-u)\det D^2 u~dx \leq \int_{\Omega}\det D^2 u~dx \leq 2^n\int_{\Omega}\det D^2 v~dx =\frac{2^n|\Omega|}{R^{2n}}\leq C(n)|\Omega|^{-1}.$$
Thus the function $u$ satisfies the conclusion of part (ii).\\
(iii) Let $\alpha=\|u\|_{L^{\infty}(\Omega)}>0$ and $v=u/\alpha$. Then, by Proposition \ref{reg_prop}, $v\in C(\overline{\Omega})\cap C^{\infty}(\Omega)$ with
$v=0$ on $\p\Omega$, and $\|v\|_{L^{\infty}(\Omega)}=1$. Furthermore, $v$ satisfies
\begin{equation}\det D^2 v= M\alpha^{p-n} |v|^p~\text{in}~\Omega.
 \label{vnp}
\end{equation}
To estimate $ M\alpha^{p-n} $ from below, we multiply both sides of the above equation by $|v|=-v$, integrate over $\Omega$  and then using (i) to obtain the desired lower bound
for $M\alpha^{p-n}$:
$$ M\alpha^{p-n} =\frac{\int_{\Omega}|v|\det D^2 v~dx }{\int_{\Omega}|v|^{p+1}~dx}\geq c(n,p)|\Omega|^{-2}.$$
Now, by (\ref{vnp}), we estimate from above the quantity
\begin{equation}M\alpha^{p-n}=\frac{\int_{B_{R/2}}\det D^2 v ~dx}{\int_{B_{R/2}} |v|^p~dx}.
 \label{anp}
\end{equation}
Recall that $B_R\subset\Omega\subset B_{nR}$. The convexity of $v$ and the fact that $v=0$ on $\p\Omega$ give for $x\in B_{R/2}$
$$|Dv(x)|\leq \frac{|v(x)|}{\dist(x,\Omega)} \leq \frac{\|v\|_{L^{\infty}(\Omega)}}{\dist(x,\Omega)}\leq C(n) R^{-1}.$$
Hence
\begin{equation}\int_{B_{R/2}}\det D^2 v ~dx = |Dv(B_{R/2})|\leq C(n) R^{-n}\leq C(n) |\Omega|^{-1}.
 \label{det_half}
\end{equation}
Applying (\ref{int_half}) to $v$, we have
\begin{equation}\int_{B_{R/2}}|v|^p~dx\geq c(n,p) |\Omega|.
 \label{int_half1}
\end{equation}
Combining (\ref{anp}), (\ref{det_half}) and (\ref{int_half1}), we obtain the desired upper bound
for $M\alpha^{p-n}$:
$$M\alpha^{p-n}\leq C(n,p)|\Omega|^{-2}.$$
\end{proof}

\begin{cor}
\label{eigen_est}
(i) Let $\Omega$ be an open, bounded, smooth and uniformly convex domain in $\R^n$. Then for the Monge-Amp\`ere eigenvalue $\lambda(\Omega)$ of (\ref{EP}), we have
  \begin{equation*}c(n) |\Omega|^{-2}\leq \lambda(\Omega)\leq C(n) |\Omega|^{-2}.
 \end{equation*}
(ii) Let $\Omega$ be an open, bounded convex domain in $\R^n$. Then for $\lambda[\Omega]$ defined by (\ref{lam_def}), we have
  \begin{equation*}c(n) |\Omega|^{-2}\leq \lambda[\Omega]\leq C(n) |\Omega|^{-2}.
 \end{equation*}
\end{cor}
\begin{proof}
 We note that (i) follows from Lemma \ref{Alem} (iii) the special case $p=n$ and $M=\lambda (\Omega)$.
 On the other hand, using the definition of $\lambda[\Omega]$ in (\ref{lam_def}), we find that (ii) follows from Lemma \ref{Alem} (i, ii) in the special case 
 $p=n+1$.
\end{proof}

\subsection{Integral estimates} 
\label{detC_sec}
In this subsection, we will prove Proposition \ref{detC}.
We start with the following integral estimate on an open, bounded, smooth, and uniformly convex domain in $\R^n$.
\begin{lem}
\label{detdet} Let $\Omega$ be an open, bounded, smooth, and uniformly convex domain in $\R^n$.
 Suppose that $u, v\in C^{1,1}(\overline{\Omega})\cap C^4 (\Omega)$ are convex functions with $u=v=0$ on $\p\Omega$. Then
$$ \int_{\Omega} |u|\det D^2 v~dx \geq \int_{\Omega} |v|(\det D^2 u)^{\frac{1}{n}} (\det D^2 v)^{\frac{n-1}{n}}~dx.$$
\end{lem}
\begin{proof}
 Let $V= (V^{ij})$ be the cofactor matrix of the Hessian matrix $D^2 v =(v_{ij})$. 
 If $D^2 v$ is positive definite then $V=(\det D^2 v) (D^2 v)^{-1}.$
 We always have $n\det D^2 v= V^{ij} v_{ij}$ where the summation convention is understood. Moreover, $V\in C^2(\Omega)$ because $v\in C^4(\Omega)$. Since $V$ is 
 divergence-free, that is, $$\p_i V^{ij}\equiv \sum_{i=1}^n\partial_i V^{ij}=0~\text{where}~\p_i=\frac{\p}{\p x_i}$$ for all $j=1,\cdots, n$, we also have
 $$n\det D^2 v= \p_i (V^{ij} v_j).$$ 
 For $\delta>0$ small, let $\Omega_{\delta}:= \{x\in\Omega: \dist (x, \p\Omega)>\delta\}.$ Then $\Omega_\delta$ is also a open smooth domain. Denote by
 $\nu=(\nu_1,\cdots,\nu_n)$ the unit outer normal vector field on $\p\Omega_{\delta}$.
 Integrating
 by parts twice, we obtain
 \begin{eqnarray*}
  \int_{\Omega_{\delta}} (-u) n\det D^2 v~dx&=&\int_{\Omega_{\delta}} (-u)\p_i(V^{ij} v_j)~dx\\&=&
   \int_{\Omega_{\delta}}  u_{i} V^{ij} v_j~dx -\int_{\p\Omega_{\delta}} u \nu_i V^{ij} v_j =
   \int_{\Omega_{\delta}}  u_{i}\p_j (V^{ij} v)~dx -\int_{\p\Omega_{\delta}} u \nu_i V^{ij} v_j\\&=&
  \int_{\Omega_{\delta}} - u_{ij} V^{ij} v~dx -\int_{\p\Omega_{\delta}} u \nu_i V^{ij} v_j + \int_{\p\Omega_{\delta}} u_i \nu_j V^{ij} v.
 \end{eqnarray*}
Letting $\delta\rightarrow 0$, using $u=v=0$ on $\p\Omega$ and $u, v\in C^{1,1}(\overline{\Omega})$, we get
\begin{equation}  \int_{\Omega} (-u) n\det D^2 v~dx= \int_{\Omega} - u_{ij} V^{ij} v~dx.
 \label{del_ineq}
\end{equation}
Using the matrix inequality
$$\trace (AB)\geq n (\det A)^{1/n} (\det B)^{1/n}~\text{for~} A, B~\text{symmetric}~\geq 0,$$
and noting that $\det V= (\det D^2 v)^{n-1}$,
we get
$$u_{ij} V^{ij}= \trace (D^2 u V)\geq n (\det D^2 u)^{1/n}(\det V)^{1/n}= n (\det D^2 u)^{\frac{1}{n}} (\det D^2 v)^{\frac{n-1}{n}}.$$
Recalling (\ref{del_ineq}) and the fact that $u, v<0$ in $\Omega$, we obtain the desired inequality.
\end{proof}
The extension of Lemma \ref{detdet} to general open bounded convex domains as stated in Proposition \ref{detC} 
will be crucial in proving the uniqueness of the Monge-Amp\`ere eigenvalue and Monge-Amp\`ere eigenfunctions up to positive multiplicative constants
on general open bounded convex domains. 

We are now ready to prove Proposition \ref{detC}.
\begin{proof}[Proof of Proposition \ref{detC}]
Let $\{\Omega_m\}\subset\Omega$ be a sequence of open, bounded, smooth and uniformly convex domains in $\Omega$ that converges to $\Omega$ in the Hausdorff distance. 
Since $u\in C^5(\Omega),$ we have $\det D^2 u\in C^3(\Omega)$.
For each $m$, the Dirichlet
problem
\begin{equation*}
 \left\{
 \begin{alignedat}{2}
   \det D^{2} u_m~&=\det D^2 u \h~&&\text{in} ~\Omega_m, \\\
u_m &=0\h~&&\text{on}~\p \Omega_m
 \end{alignedat}
 \right.
\end{equation*}
has a unique solution $u_m\in C^{4+\alpha}(\overline{\Omega_m})$ for all $\alpha\in (0,1)$ by the classical result of Caffarelli-Nirenberg-Spruck \cite{CNS}; see Theorem 1 and Remark 2 in \cite{CNS}.

Similarly, for each $m$, the Dirichlet
problem
\begin{equation*}
 \left\{
 \begin{alignedat}{2}
   \det D^{2} v_m~&=\det D^2 v \h~&&\text{in} ~\Omega_m, \\\
v_m &=0\h~&&\text{on}~\p \Omega_m
 \end{alignedat}
 \right.
\end{equation*}
has a unique solution $v_m\in C^{4+\alpha}(\overline{\Omega_m})$ for all $\alpha\in (0, 1)$.

Applying Lemma \ref{detdet} to $u_m, v_m$ and $\Omega_m$, we have
$$ \int_{\Omega_m} |u_m|\det D^2 v_m~dx \geq \int_{\Omega_m} |v_m|(\det D^2 u_m)^{\frac{1}{n}} (\det D^2 v_m)^{\frac{n-1}{n}}~dx.$$
It follows that
\begin{equation} \int_{\Omega_m} |u_m|\det D^2 v~dx \geq \int_{\Omega_m} |v_m|(\det D^2 u)^{\frac{1}{n}} (\det D^2 v)^{\frac{n-1}{n}}~dx.
 \label{detdetm}
\end{equation}
We will let $m\rightarrow\infty$ in (\ref{detdetm}) to obtain (\ref{term0}). To see this, we first show that
\begin{equation}\label{term1}\int_{\Omega_m} |v_m|(\det D^2 u)^{\frac{1}{n}} (\det D^2 v)^{\frac{n-1}{n}}~dx\rightarrow \int_{\Omega} |v|(\det D^2 u)^{\frac{1}{n}} (\det D^2 v)^{\frac{n-1}{n}}~dx ~
\text{when} ~m\rightarrow \infty.
\end{equation}
Note that
\begin{multline} \int_{\Omega} |v|(\det D^2 u)^{\frac{1}{n}} (\det D^2 v)^{\frac{n-1}{n}}~dx -  \int_{\Omega_m} |v_m|(\det D^2 u)^{\frac{1}{n}} (\det D^2 v)^{\frac{n-1}{n}}~dx
\\= \int_{\Omega\setminus \Omega_m} |v|(\det D^2 u)^{\frac{1}{n}} (\det D^2 v)^{\frac{n-1}{n}}~dx  +
 \int_{\Omega_m}[|v|-|v_m|](\det D^2 u)^{\frac{1}{n}} (\det D^2 v)^{\frac{n-1}{n}}~dx\\ \equiv A_m + B_m.
\label{term1m}
\end{multline}
We will show that $A_m$ and $B_m$ tend to $0$ when $m\rightarrow \infty$.
Indeed, since $v\in C(\overline{\Omega})$, $v=0$ on $\p\Omega$, and  $\Omega_m$ converges to $\Omega$ in the Hausdorff distance, we have
 \begin{equation}\|v\|_{L^{\infty}(\Omega\setminus \Omega_m)}\rightarrow 0 ~\text{when} ~m\rightarrow \infty.
\label{uzerom}
\end{equation}
Therefore, from (\ref{uzerom}) and the bound  $\int_{\Omega}(\det D^2 u)^{\frac{1}{n}}  (\det D^2 v)^{\frac{n-1}{n}}~dx\leq M$, we find 
that $A_m$ goes to $0$
when $m\rightarrow \infty$. 

To show that $B_m$ goes to $0$ when $m\rightarrow \infty$, it suffices to show that
$$\|v-v_m\|_{L^{\infty}(\Omega_m)}\rightarrow 0.$$
Note that, on $\p\Omega_m$,
$$-\|v-v_m\|_{L^{\infty}(\p \Omega_m)} + v_m \leq v\leq \|v-v_m\|_{L^{\infty}(\p \Omega_m)} + v_m.$$
In $\Omega_m$, we have $$\det D^2 v=\det D^2 (v_m \pm \|v-v_m\|_{L^{\infty}(\p \Omega_m)}).$$ Thus, 
by the comparison principle, Lemma \ref{comp_prin}, we have in $\Omega_m$
$$-\|v-v_m\|_{L^{\infty}(\p \Omega_m)} + v_m \leq v\leq \|v-v_m\|_{L^{\infty}(\p \Omega_m)} + v_m.$$
It follows that 
$$\|v-v_m\|_{L^{\infty}(\Omega_m)}\leq \|v-v_m\|_{L^{\infty}(\p\Omega_m)}= \|v\|_{L^{\infty}(\p\Omega_m)}\rightarrow 0~\text{when} ~m\rightarrow \infty$$
by (\ref{uzerom}). Consequently, (\ref{term1}) holds.

Similarly, we also have
\begin{equation}\label{term2}\int_{\Omega_m} |u_m|\det D^2 v~dx\rightarrow \int_{\Omega} |u|\det D^2 v~dx ~\text{when} ~m\rightarrow \infty.
\end{equation}
Hence, using (\ref{term1}) and (\ref{term2}), we can let $m\rightarrow\infty$ in (\ref{detdetm}) to obtain the desired inequality (\ref{term0}).
\end{proof}

\section{Monge-Amp\`ere equations with degenerate right hand side}
\label{D_sec}
\subsection{An extension of Tso's theorem}  Let $\Omega$ be an open, bounded convex domain in $\R^n$.
For $p\geq 0$ and $\lambda>0$, consider the functional
\begin{equation}J_{p,\lambda}(u,\Omega)=\frac{1}{n+1}\int_{\Omega} (-u)\det D^2 u~dx -\frac{\lambda}{p+1}\int_{\Omega}(-u)^{p+1}~dx
 \label{J_func}
\end{equation}
over all convex functions $u\in C(\overline{\Omega})$ with $u=0$ on $\p\Omega$.

When the domain $\Omega$ is clear from the context, we can write $J_{p,\lambda}(u)$ for $J_{p,\lambda}(u,\Omega)$.

We recall the following theorem due to Tso \cite[Corollary 4.2 and Theorem E]{Tso}.
\begin{thm} [Tso's theorem]
\label{sub_thm0}
 Let $\Omega$ be an open, bounded, smooth and uniformly convex domain in $\R^n$. Then, for each $0\leq p\neq n$, there exists a nonzero convex solution $u\in C^{0, 1}(\overline{\Omega})\cap C^{\infty}(\Omega)$ to the Dirichlet problem
\begin{equation*}
 \left\{
 \begin{alignedat}{2}
   \det D^{2} u~&=|u|^{p} \h~&&\text{in} ~\Omega, \\\
u &=0\h~&&\text{on}~\p \Omega.
 \end{alignedat}
 \right.
\end{equation*}
Moreover, if $0\leq p<n$ then the nonzero convex function $u$ is unique and it minimizes the functional 
$J_{p, 1}(u,\Omega)$
over all convex functions $u\in C(\overline{\Omega})\cap C^2(\Omega)$ with $u=0$ on $\p\Omega$ and having positive definite Hessian $D^2 u$ in $\Omega$.
 \end{thm}
When $0\leq p<n$, Theorem \ref{sub_thm0} was extended by Hartenstine \cite[Theorem 3.1]{Har2} to the case where the domain $\Omega$ is only assumed to be bounded and strictly convex.
In the following theorem, we further extend Theorem \ref{sub_thm0} to case where the domain $\Omega$ is only assumed to be bounded and convex. Our extension works for all $p$.
\begin{thm} 
\label{sub_thm}
 Let $\Omega$ be an open, bounded convex domain in $\R^n$. Then, for each $0\leq p\neq n$, there exists a nonzero convex Aleksandrov solution $u\in C(\overline{\Omega})
 \cap C^{\infty}(\Omega)$ to the Dirichlet problem
\begin{equation}
 \left\{
 \begin{alignedat}{2}
   \det D^{2} u~&=|u|^{p} \h~&&\text{in} ~\Omega, \\\
u &=0\h~&&\text{on}~\p \Omega.
 \end{alignedat}
 \right.
 \label{DMA}
\end{equation}
Moreover, if $0\leq p<n$ then the nonzero convex function  $u$ is unique and it minimizes the functional 
$$J_{p, 1}(u,\Omega)=\frac{1}{n+1}\int_{\Omega} (-u)\det D^2 u~dx -\frac{1}{p+1}\int_{\Omega}(-u)^{p+1}~dx$$
over all convex functions $u\in C(\overline{\Omega})$ with $u=0$ on $\p\Omega$.
 \end{thm}
 \begin{rem}
 The conclusions of Theorem \ref{sub_thm} remain unchanged if we replace $1$ by $\lambda>0$ in the functional $J_{p,1}(\cdot,\Omega)$ and the equation $\det D^2 u=|u|^p$
is replaced by $\det D^2 u=\lambda |u|^p$.
\end{rem}

For reader's convenience, we include below a simple proof of Theorem \ref{sub_thm} for general open bounded convex domains $\Omega$. Our proof adapts arguments from \cite{Har2}.

We start with following compactness result.
\begin{prop}
\label{sub_comp} Let $0\leq p\neq n$.  Let $\Omega$ be an open, bounded convex domain in $\R^n$.  Let $\{\Omega_m\}$ be a sequence of open, bounded, smooth and uniformly convex domains in $\R^n$ 
  that converges to $\Omega$ in the Hausdorff distance. For each $m$, 
consider a
nonzero convex solution $u_m \in C^{0, 1}(\overline{\Omega_m})\cap C^{\infty}(\Omega_m)$
to 
\begin{equation}
 \left\{
 \begin{alignedat}{2}
  \det D^{2} u_m~&=|u_m|^{p} \h~&&\text{in} ~\Omega_m, \\\
u_m &=0\h~&&\text{on}~\p \Omega_m.
 \end{alignedat}
 \right.
\label{DMAm}
\end{equation}
Then up to extracting a subsequence, $\{u_m\}$
converges uniformly on compact subsets of $\Omega$ to a nonzero convex Aleksandrov solution $u\in C(\overline{\Omega})\cap C^{\infty}(\Omega)$ of (\ref{DMA}). Furthermore, $J_{p, 1}(u_m,\Omega_m)\rightarrow J_{p,1}(u,\Omega).$
\end{prop}

\begin{proof}
By Lemma \ref{Alem}(iii), we have
\begin{equation}c(n, p) |\Omega_m|^{\frac{2}{n-p}}\leq \|u_m\|_{L^{\infty}(\Omega_m)} \leq C(n, p) |\Omega_m|^{\frac{2}{n-p}}.
 \label{um_est}
\end{equation}
Applying the Aleksandrov maximum principle, Theorem \ref{Alex_thm}, to the solution $u_m$ of (\ref{DMAm}), and using the bound (\ref{um_est}), we find
\begin{eqnarray}
 \label{1n_Holder}
 |u_m(x)|^n&\leq& C(n) (\diam \Omega_m)^{n-1} \dist(x,\p\Omega_m) 
 \int_{\Omega_m} \det D^2 u_m~dx\nonumber \\ & \leq& C(n, p, \diam \Omega_m, |\Omega_m|) \dist(x,\p\Omega_m)~\forall x\in\Omega_m.
\end{eqnarray}
Due to the convergence of $\Omega_m$ to $\Omega$ in the Hausdorff distance, the sequence $\{C(n, p, \diam \Omega_m, |\Omega_m|)\}$ is bounded by a constant $C(n,p,\Omega)$ independent of $m$.
Now, combining (\ref{1n_Holder}) with $u_m=0$ on $\p\Omega_m$, we find that $\{u_m\}$ is uniformly $C^{0,\frac{1}{n}}$ at the boundary:
$$|u_m(x)-u_m(y)|\leq C(n,p,\Omega)|x-y|^{\frac{1}{n}}~\forall x\in\Omega_m,~\forall y\in\p\Omega_m.$$
From the convexity of $u_m$,
we find that the functions $u_m$ have uniformly bounded global $C^{0, \frac{1}{n}}$ norm on $\overline{\Omega_m}$. By the Arzela-Ascoli theorem, there 
exists a subsequence of $\{u_m\}$, still
denoted by $\{u_m\}$, that converges locally uniformly to a convex function $u$ on $\Omega$.
The compactness theorem
of the Monge-Amp\`ere equation, Theorem \ref{compact_thm}, then gives that the function $u$ is actually an Aleksandrov solution of (\ref{DMA}). The 
estimates (\ref{um_est}) show that $u\not\equiv 0$. By Proposition \ref{reg_prop}, $u\in C^{\infty}(\Omega)$.

Because $u_m$ and $u$ solve (\ref{DMAm}) and (\ref{DMA}), respectively, 
we have
$$J_{p, 1}(u_m,\Omega_m)= \frac{p-n}{(n+1)(p+1)}\int_{\Omega_m}|u_m|^{p+1}~ dx~
\text{and}~
J_{p,1}(u,\Omega)= \frac{p-n}{(n+1)(p+1)}\int_{\Omega}|u|^{p+1}~ dx.$$
Since it is easy to see that
$$\int_{\Omega_m} |u_m|^{p+1} ~dx\rightarrow \int_{\Omega} |u|^{p+1}~dx,$$
the convergence $J_{p, 1}(u_m,\Omega_m)\rightarrow J_{p,1}(u,\Omega)$ follows.
 \end{proof}
 Next, we  prove a uniqueness result for (\ref{DMA}) when $0\leq p<n$ which follows from the interior regularity of its solutions and a rescaling argument.
 \begin{prop}
 \label{sub_prop} Let $\Omega$ be an open, bounded convex domain in $\R^n$. Let $p\in [0, n)$.
 Then there is a 
  unique nonzero convex Aleksandrov solution $u\in C(\overline{\Omega})$ to (\ref{DMA}). 
 \end{prop}

\begin{proof} 
 Let $u\in C(\overline{\Omega})$ be a nonzero convex solution to (\ref{DMA}) whose exsitence is guaranteed by Proposition \ref{sub_comp}. We show that it is unique.
By Proposition \ref{reg_prop}, $u\in C^{\infty}(\Omega)$ and $u$ is strictly convex in $\Omega$. The uniqueness of $u$ follows as in 
Tso \cite[Proposition 4.1]{Tso}. Since our setting is slightly different, we include the proof for reader's convenience. Suppose we have two nonzero convex
solutions $u$ and $v$ to (\ref{DMA}) with $u-v$ being positive somewhere in $\Omega$.  By translation of coordinates, we can assume that $x_0=0\in\Omega$ satisfies
$u(0)- v(0)=\max_{x\in \overline{\Omega}} (u-v)(x) =\delta>0.$\\
\noindent
For $1<\Lambda\leq 2$, consider for $x\in\Omega$, $u_{\Lambda}(x)= u(x/\Lambda)$
and $$\eta_{\Lambda}(x) = v(x)/u_{\Lambda}(x).$$ If $\dist(x,\p\Omega)\rightarrow 0$, then $\eta_{\Lambda}(x)\rightarrow 0.$
Note that $u, v<0$ in $\Omega$ and hence $$\eta_\Lambda(0) = v(0)/u(0) =[u(0)-\delta]/u(0) \geq 1+\e~ \text{for some }\e>0.$$
Therefore, the function $\eta_{\Lambda}$ attains its maximum value at $x_{\Lambda}\in\Omega$ with $\eta_{\Lambda}(x_{\Lambda})\geq 1+\e.$
At $x=x_{\Lambda}$, using $u, v\in C^{\infty}(\Omega)$, $\left(D^2 \eta_{\Lambda }(x_{\Lambda})\right)\leq 0$ and $u_{\Lambda}(x_{\Lambda})\leq 0$, we can compute
$$D^2 v= \eta_{\Lambda} D^2 u_{\Lambda} + D^2 \eta_{\Lambda } u_{\Lambda}\geq \eta_{\Lambda} D^2 u_{\Lambda}.$$
It follows that
$$|v(x_{\Lambda})|^p =\det D^2 v(x_{\Lambda}) \geq \eta_{\Lambda}^n\det D^2 u_{\Lambda}(x_{\Lambda}) =\eta_{\Lambda}^n \Lambda^{-2n}
\det D^2 u (x_{\Lambda}/\Lambda) = \eta_{\Lambda}^n \Lambda^{-2n} |u(x_{\Lambda}/\Lambda)|^p.$$
Using the maximality of $\eta_{\Lambda}$ at $x_{\Lambda}$, we find
$$1\geq \eta_{\Lambda}^n \Lambda^{-2n} \left [|u(x_{\Lambda}/\Lambda)|/|v(x_{\Lambda})|\right]^p= \eta_{\Lambda}^{n-p}(x_{\Lambda}) \Lambda^{-2n} \geq (1+\e)^{n-p}\Lambda^{-2n}.$$
Letting $\Lambda\searrow 1$, using $p<n$ and $\e>0$, we obtain a contradiction. Thus $u$ must be the unique nonzero convex solution to (\ref{DMA}) and the proof of the proposition is complete.
\end{proof}

We recall the following approximation lemmas; see Lemmas 3.1 and 3.2 in \cite{Har2}.
\begin{lem}
\label{lem31Har}
Let $p\in [0, n)$. Let $\Omega$ be an open, bounded uniformly convex domain and let $u\in C(\overline{\Omega})$ be a convex function with $u=0$ on $\p\Omega$ such that $J_{p, 1}(u,\Omega)<\infty.$ Then for any $\e>0$, there exists 
 a convex function $v\in C(\overline{\Omega})\cap C^2(\Omega)$ with $v=0$ on $\p\Omega$ and positive-definite Hessian such that $|J_{p, 1}(v,\Omega)-J _{p, 1}(u,\Omega)|<\e.$
\end{lem}
\begin{lem}
\label{lem32Har}
Let $p\in [0, n)$. Let $\Omega$ be an open, bounded convex domain, and suppose $v\in C(\overline{\Omega})$ is a convex function with $v=0$ on $\p\Omega$ satisfying $J_{p, 1}(v,\Omega)<\infty.$ 
Let $\{\Omega_m\}$ be a sequence of open, bounded, smooth and uniformly convex domains in $\R^n$, satisfying $\Omega_k\supset\Omega_{k+1}\supset\Omega$ for all $k$, and such that
$\Omega_m$ converges to $\Omega$ in the Hausdorff distance. Then for every $\e>0$, there exist an $m$ and 
 a convex function $w\in C(\overline{\Omega_m})$ with $w=0$ on $\p\Omega_m$ such that $|J_{p, 1}(v,\Omega)-J _{p, 1}(w,\Omega_m)|<\e.$
\end{lem}
Now, we can give a proof of Theorem \ref{sub_thm}.
\begin{proof}[Proof of Theorem \ref{sub_thm}] Fix $0\leq p\neq n$.
Let $\{\Omega_m\}$ be a sequence of open, bounded, smooth and uniformly convex domains in $\R^n$, satisfying $\Omega_k\supset\Omega_{k+1}\supset\Omega$ for all $k$, and such that
$\Omega_m$ converges to $\Omega$ in the Hausdorff distance. Let $u_m$ and $u$ be as in Proposition \ref{sub_comp}. Then, $u\in C(\overline{\Omega})
\cap C^{\infty}(\Omega)$ is a nonzero convex Aleksandrov solution to
the Dirichlet problem (\ref{DMA}).
From the uniqueness result in Proposition \ref{sub_prop} for the case $0\leq p<n$, it remains to show that  if $0\leq p<n$ then
$u$ minimizes $J_{p, 1}(\cdot,\Omega)$ over all convex functions $w\in C(\overline{\Omega})$ with $w=0$ on $\p\Omega$. 

Indeed, for each $m$, the nonzero convex solution $u_m$ to (\ref{DMAm}), as given by \cite[Corollary 4.2]{Tso}, is the minimizer of $J_{p, 1}(\cdot,\Omega_m)$ over 
all functions $v\in C(\overline{\Omega_m})\cap C^2(\Omega_m)$ satisfying $v=0$ on $\p\Omega_m$ and whose Hessian is positive definite at all points in $\Omega_m$. By Lemma 
\ref{lem31Har}, we find that $u_m$ is also 
the minimizer of $J_{p, 1}(\cdot,\Omega_m)$ over all functions $w\in C(\overline{\Omega_m})$ satisfying $w=0$ on $\p\Omega_m$. 

We use the above fact together with the convergence $J_{p, 1}(u_m,\Omega_m)\rightarrow J_{p,1}(u,\Omega)$ proved in Proposition \ref{sub_comp}
to establish the minimality of $u$. Suppose that this is not the case. Then there exists a convex function $v\in C(\overline{\Omega})$ with $v=0$ on $\p\Omega$ such that $\e:=J_{p, 1}(u,\Omega)-J_{p, 1}(v,\Omega)>0.$ Then, for sufficiently large $k$, we must have $|J_{p, 1}(u_k,\Omega_k)-J_{p,1}(u,\Omega)|<\e/3$ and also by Lemma \ref{lem32Har}, for $k$ large enough, there is
a convex function $v_k\in C(\overline{\Omega_k})$ with $v_k=0$ on $\p\Omega_k$ such that $|J_{p, 1}(v,\Omega)-J _{p, 1}(v_k,\Omega_k)|<\e/3.$ It follows that 
there exists $k$ such that $J_{p, 1}(v_k,\Omega_k)<J_{p, 1} (u_k, \Omega_k)-\e/3.$ This contradicts the minimality property of $u_k$. The proof of the theorem is complete.
\end{proof}

\subsection{Degenerate Monge-Amp\`ere equations and Abreu's equation} We digress a bit in this subsection to discuss another context where
the degenerate Monge-Amp\`ere equations (\ref{DMA}) appear. It turns that (\ref{DMA}) also appear in the analysis of the Abreu equation.
The Abreu equation \cite{Ab} is a fourth order fully nonlinear partial differential equation of the form
\begin{equation}
\label{AMCE}
 U^{ij}w_{ij} =-f,~w = [\det D^{2} u]^{-1}\quad \text{in}~\Omega,
\end{equation}
where $u$ is a locally uniformly convex function defined in $\Omega\subset\R^{n}$, and $U = (U^{ij})$ is the matrix of cofactors of the 
Hessian matrix $D^{2}u$ of $u$.

When $\Omega$ is the interior of a bounded polytope (that is the intersection of finitely many closed half-spaces), equation (\ref{AMCE}) arises in the 
study of the existence of constant scalar curvature K\"ahler metrics for toric varieties \cite{D1} in complex geometry and the function $f$ corresponds to 
the scalar curvature of toric varieties. In this case, $\det D^2 u$ blows up, or equivalently, the inverse of the Hessian determinant $w$ vanishes, near the boundary $\p\Omega$. 
In obtaining (essentially sharp) blow up
rate of $\det D^2 u$ near the boundary $\p\Omega$, Donaldson \cite[Theorem 5]{D2} uses the maximum principle and an interesting Monge-Amp\`ere differential
 inequality  concerning negative 
 convex functions $R$ on
 $\Omega$ with 
 $$\lim_{\dist(x, \partial \Omega)\rightarrow 0} \frac{R(x)}{\dist(x,\p\Omega)}= -\infty,~\text{and}~\det D^2 R \geq |R|^{n-1}. $$
This study leads to (\ref{DMA}) with $p=n-1$ on open bounded convex polytopes; see \cite[p. 123]{D2}.

\section{Proof of Theorem \ref{lam_thm}}
\label{sec_proof}
In this section, we prove Theorem \ref{lam_thm} by proving Propositions \ref{exi_prop}, \ref{alm_Lip}, \ref{uni_lam} and \ref{lam_nice}.
\subsection{Existence of the Monge-Amp\`ere eigenfunctions}
The existence of nonzero convex solutions to (\ref{EVP_eq}) with $\lambda=\lambda[\Omega]$ defined by (\ref{lam_def}) follows from the following proposition.
\begin{prop}
\label{exi_prop}
Let $\Omega$ be an open bounded convex domain in $\R^n$.
Let $\lambda=\lambda[\Omega]$ be defined by (\ref{lam_def}). Then (\ref{lam_est}) holds and there exists a  nonzero convex solution 
$u\in C(\overline{\Omega})\cap C^{\infty}(\Omega)$ to (\ref{EVP_eq}).
\end{prop}
\begin{proof}
By Corollary \ref{eigen_est}, we have (\ref{lam_est}). For the existence of nonzero convex solutions to (\ref{EVP_eq}), we argue as in Tso \cite[Corollary 4.3]{Tso}.
For each $0\leq p<n$, by Theorem \ref{sub_thm}, there exists a unique nonzero convex solution $u_p\in C(\overline{\Omega})\cap C^{\infty}(\Omega)$ to
\begin{equation*}
 \left\{
 \begin{alignedat}{2}
   \det D^{2} u~&=\lambda|u|^{p} \h~&&\text{in} ~\Omega, \\\
u &=0\h~&&\text{on}~\p \Omega.
 \end{alignedat}
 \right.
\end{equation*}
Moreover, $u_p$ minimizes the functional $J_{p,\lambda}(\cdot)= J_{p,\lambda}(\cdot,\Omega)$, defined 
by (\ref{J_func}), over all convex functions $u\in C(\overline{\Omega})$ with $u=0$ on $\p\Omega$. 
Our proof, however, does not use the uniqueness property of $u_p$.

We will bound $\|u_p\|_{L^{\infty}(\Omega)}$ uniformly (in $p$) from above and below. Then, 
we can argue as in Proposition \ref{sub_comp} to show that, up to extracting a subsequence, $\{u_p\}$ converges, as $p\nearrow n$,  
 uniformly on 
 compact subsets of $\Omega$ to a nonzero convex solution $u\in C(\overline{\Omega})\cap C^{\infty}(\Omega)$ of (\ref{EVP_eq}).

{\it Step 1 (Bound $\|u_p\|_{L^{\infty}(\Omega)}$ from above).} Let $\alpha_p =\|u_p\|_{L^{\infty}(\Omega)}$ and $v_p= u_p/\alpha_p.$ Then $v_p\in C(\overline{\Omega})
\cap C^{\infty}(\Omega)$
is convex with $v_p=0$ on $\p\Omega$ and $\|v_p\|_{L^{\infty}(\Omega)}=1$. 
From the equation for $u_p$, we deduce that
$$\alpha_p^{n-p}\det D^2 v_p = \lambda |v_p|^p~\text{in}~\Omega.$$
Multiplying both sides by $|v_p|$ and integrating over $\Omega$, we get
\begin{equation}\displaystyle\alpha_p = \left[ \frac{\lambda \int_{\Omega}|v_p|^{p+1}~dx}{\int_{\Omega}|v_p|\det D^2 v_p~dx}\right]^{\frac{1}{n-p}}.
 \label{al_eq}
\end{equation}
It follows that
\begin{equation}
J_{p,\lambda}(u_p)= J_{p,\lambda}(\alpha_p v_p)= \lambda \alpha_p^{p+1} \|v_p\|^{p+1}_{L^{p+1}(\Omega)} (p-n)(n+1)^{-1}(p+1)^{-1}.
 \label{upnorm}
\end{equation}
By the H\"older inequality, we have
$$\int_{\Omega} |v_p|^{p+1}~dx = \|v_p\|^{p+1}_{L^{p+1}(\Omega)}\leq \|v_p\|^{p+1}_{L^{n+1}(\Omega)} |\Omega|^{\frac{n-p}{n+1}}.$$
Thus, using (\ref{al_eq}), the definition of $\lambda$ in (\ref{lam_def}) and Lemma \ref{Alem}, we find
$$\alpha_p \leq  \left[ \frac{\lambda \int_{\Omega}|v_p|^{n+1}~dx}{\int_{\Omega}|v_p|\det D^2 v_p~dx}\right]^{\frac{1}{n-p}}
 \left[\frac{|\Omega|}{\int_{\Omega}|v_p|^{n+1}~dx}\right]^{\frac{1}{n+1}} \leq \left[\frac{|\Omega|}{\int_{\Omega}|v_p|^{n+1}~dx}\right]^{\frac{1}{n+1}}
 \leq C(n).
$$
{\it Step 2 (Bound $\|u_p\|_{L^{\infty}(\Omega)}$ from below).} This step uses the minimality property of $u_p$ with respect to the functional $J_{p,\lambda}(\cdot,\Omega)$ over 
all convex
functions $w\in C(\overline{\Omega})$ with $w=0$ on $\p\Omega$. We will use
$$J_{p,\lambda}(u_p)\leq J_{p,\lambda} (\alpha v)$$
for a special positive number $\alpha$ and a convex function $v\in C(\overline{\Omega})$ 
with $v=0$ on $\p\Omega$ and $\|v\|_{L^{\infty}(\Omega)}=1$. Motivated by the calculation in {\it Step 1}, once $v$ is determined, we can choose
$$\alpha = \left[ \frac{\lambda \int_{\Omega}|v|^{p+1}~dx}{\int_{\Omega}|v|\det D^2 v~dx}\right]^{\frac{1}{n-p}}.$$
Since $\|v\|_{L^{\infty}(\Omega)}=1$, $$\alpha\geq \left[ \frac{\lambda \int_{\Omega}|v|^{n+1}~dx}{\int_{\Omega}|v|\det D^2 v~dx}\right]^{\frac{1}{n-p}}.$$
By the definition of $\lambda$ in (\ref{lam_def}), we can choose a convex function $v\in C(\overline{\Omega})$ with $v=0$ on $\p\Omega$ and $\|v\|_{L^{\infty}(\Omega)}=1$ so that $\alpha\geq \frac{1}{2}$.
Then $\|v\|_{L^{p+1}(\Omega)}^{p+1}\geq \|v\|_{L^{n+1}(\Omega)}^{n+1}\geq A$ where 
$$ A:=\inf\left\{\int_{\Omega} |u|^{n+1}~dx: u \text{ is convex in } \Omega, u\in C(\overline{\Omega}), u=0~\text{on}~\p\Omega,~\|u\|_{L^{\infty}(\Omega)}=1\right\}$$
and the minimality of $u_p$ now leads to
\begin{eqnarray*}J_{p,\lambda}(u_p) \leq J_{p,\lambda} (\alpha v)&=& 
\lambda \alpha^{p+1} \|v\|_{L^{p+1}(\Omega)}^{p+1} (n+1)^{-1} (p+1)^{-1} (p-n)\\ &\leq& \lambda 2^{-p-1} A(n+1)^{-1} (p+1)^{-1} (p-n)<0.
\end{eqnarray*}
Recalling (\ref{upnorm}), we have for $\alpha_p=\|u_p\|_{L^{\infty}(\Omega)} $ and $\|v_p\|_{L^{\infty}(\Omega)}=1$, 
$$\lambda \alpha_p^{p+1} \|v_p\|^{p+1}_{L^{p+1}(\Omega)} (p-n)(n+1)^{-1}(p+1)^{-1}\leq \lambda 2^{-p-1} A(n+1)^{-1} (p+1)^{-1} (p-n)<0.$$
Hence, using  $\|v_p\|_{L^{p+1}(\Omega)}\leq |\Omega|^{\frac{1}{p+1}}$ which follows from $\|v_p\|_{L^{\infty}(\Omega)}=1$, and then invoking Lemma \ref{Alem}, we obtain
$$\|u_p\|_{L^{\infty}(\Omega)}=\alpha_p \geq 2^{-1} A^{\frac{1}{p+1}}\|v_p\|^{-1}_{L^{p+1}(\Omega)}\geq 2^{-1} A^{\frac{1}{p+1}} |\Omega|^{-\frac{1}{p+1}}\geq 2^{-1}[c(n)]^{\frac{1}{p+1}}\geq 2^{-1}[c(n)]^{\frac{1}{n+1}}.$$
We have established the uniform bound for $\|u_p\|_{L^{\infty}(\Omega)}$ from above and below, and thus completing the proof of the proposition.
\end{proof}
The following proposition gives another way to construct, via approximation, nonzero convex solutions to (\ref{EVP_eq}) with $\lambda$ replaced by 
a (possible different) positive constant $\Lambda$.
\begin{prop}
\label{Lam_prop}
Let $\Omega$ be an open bounded convex domain in $\R^n$.
Let $\{\Omega_m\}$ be a sequence of open, bounded, smooth and uniformly convex domains in $\R^n$ 
that converges to $\Omega$ in the Hausdorff distance. For each $m$, we
take $u_m\in C^{1,1}(\overline{\Omega_m})\cap C^{\infty}(\Omega_m)$ with $\|u_m\|_{L^{\infty}(\Omega_m)}=1$ to be 
a nonzero convex eigenfunction of the Monge-Amp\`ere operator on $\Omega_m$:
\begin{equation*}
\left\{
 \begin{alignedat}{2}
   \det D^2 u_m~& = \lambda(\Omega_m)|u_m|^n~&&\text{in} ~  \Omega_m, \\\
u_m &= 0~&&\text{on}~ \p\Omega_m.
 \end{alignedat} 
  \right.
  \end{equation*}
Then we can find a subsequence of $\{m\}_{m=1}^{\infty}$, still denoted by $\{m\}_{m=1}^{\infty}$, such that  
$$\lim_{m\rightarrow \infty}\lambda(\Omega_m)=\Lambda>0$$
and $\{u_m\}$ converges uniformly on compact subsets of $\Omega$ to a nonzero convex solution of
\begin{equation}
\label{EPLam}
\left\{
 \begin{alignedat}{2}
   \det D^2 u~& = \Lambda|u|^n~&&\text{in} ~  \Omega, \\\
u &= 0~&&\text{on}~ \p\Omega.
 \end{alignedat} 
  \right.
  \end{equation}
\end{prop}
\begin{proof}
 By Corollary \ref{eigen_est}, we have the estimates
 $$c(n) |\Omega_m|^{-2}\leq \lambda(\Omega_m)\leq C(n)|\Omega_m|^{-2}.$$
 Thus, the sequence $\{\lambda(\Omega_m)\}$ is uniformly bounded from above and below. Thus, up to extracting a subsequence, $\{\lambda(\Omega_m)\}$ converges to a positive
 constant $\Lambda>0$. Since $\|u_m\|_{L^{\infty}(\Omega_m)}=1$, we can argue as in Proposition \ref{sub_comp} to show that, up to extracting a subsequence further, $\{u_m\}$ converges 
 uniformly on 
 compact subsets of $\Omega$ to a nonzero convex solution of (\ref{EPLam}).
\end{proof}
\subsection{Global almost Lipschitz property of the Monge-Amp\`ere eigenfunctions} In this subsection, 
we prove that the Monge-Amp\`ere eigenfunctions in Proposition \ref{exi_prop} are almost Lipschitz globally, that is 
$u\in C^{0,\beta}(\overline{\Omega})$ for all $\beta\in (0, 1)$.  
\begin{prop}
 \label{alm_Lip}
 Let $\Omega$ be an open bounded convex domain in $\R^n$.
 Let  $u\in C(\overline{\Omega})$ be a nonzero convex eigenfunction of (\ref{EVP_eq}) with $\lambda=\lambda[\Omega]$ defined by (\ref{lam_def}). Then,
 for all $\beta\in (0, 1)$, we have $u\in C^{0,\beta}(\overline{\Omega})$ and the estimate
 \begin{equation}
 \label{ualmLip}
 |u(x)|\leq C(n, \beta, \emph{diam} \Omega) [\emph{dist}(x,\p\Omega) ]^{\beta}\|u\|_{L^{\infty}(\Omega)}~\text{for all}~x\in\Omega.
 \end{equation}
 \end{prop}
 \begin{proof}
 By the convexity of $u$, the global regularity $u\in C^{0,\beta}(\overline{\Omega})$ for all $\beta\in (0, 1)$ follows from the boundary estimate (\ref{ualmLip}). It remains 
 to establish this estimate. Without loss of generality, we can assume that
 $\|u\|_{L^{\infty}(\Omega)}=1.$
 
Let $x^{\ast}$ be an arbitrary point in $\Omega$. By translation and rotation of coordinates, we can assume that we have the following geometric setting: the origin $0$ of $\R^n$ lies on $\p\Omega$,  
 the $x_n$-axis points inward $\Omega$, $x^{\ast}$ lies on the $x_n$-axis, and the minimum distance to the boundary of $\Omega$ from $x^{\ast}$ is achieved at the origin. We denote a point $x= (x_1, \cdots, x_n)\in\R^n$
 by $(x', x_n)$ where $x'=(x_1, \cdots, x_{n-1})$. 
 
 To prove (\ref{ualmLip}), it suffices to prove that for all $x\in\Omega$, we have
 \begin{equation}
 \label{betapoly2}
 |u(x)|\leq C(n,\alpha, \diam \Omega) x_n^{\alpha} ~\text{for all}~\alpha\in (0, 1). 
 \end{equation}
 First, we 
prove a weaker bound
\begin{equation} |u(x)| \leq C(n, \diam \Omega) x_n^{\frac{2}{n+1}}~\text{for all}~x\in\Omega.
\label{ustep1}
\end{equation}
Indeed,  motivated by \cite[Lemma 1]{C1},  let us consider, for $\alpha\in (0, 1)$ and $x\in\Omega$,
\begin{equation}\phi_{\alpha}(x)= x_n^{\alpha} (|x'|^2 -C_\alpha)~\text{where}~C_{\alpha}= \frac{1 +2 [\diam\Omega]^2}{\alpha (1-\alpha)}.
\label{alphaeq}
\end{equation}
Then, denoting the second partial derivative operator  $\frac{\p^2}{\p x_i \p x_j}$ by $D_{ij}$, we have
$$D_{ij}\phi_{\alpha} = 2x_n^{\alpha}\delta_{ij}~\text{for}~i, j\leq n-1;~ D_{in}\phi_{\alpha}= 2\alpha x_i x_n^{\alpha-1}~ \text{and}~ D_{nn} \phi_{\alpha} = \alpha (\alpha-1) (|x'|^2-C_\alpha)x_n^{\alpha-2}.$$
Here, $\delta_{ij}=1$ if $i=j$ and $\delta_{ij}=0$ if $i\neq j$.
We can compute for all $x\in\Omega$,
$$\det D^2 \phi_{\alpha} (x)= 2^{n-1}x_n^{n\alpha-2} [\alpha(1-\alpha) C_\alpha- (\alpha^2+ \alpha) |x'|^2]\geq 2^{n-1}x_n^{n\alpha-2}.$$
Therefore, $\phi_{\alpha}$ is convex in $\Omega$ with
\begin{equation}\det D^2\phi_{\alpha}(x)\geq x_n^{n\alpha-2}~\text{in}~\Omega~ \text{and} ~\phi_{\alpha}\leq 0~\text{on}~\p\Omega.
\label{detphi}
\end{equation}
To prove (\ref{ustep1}), we choose $\alpha =\frac{2}{n+1}$ in (\ref{alphaeq}) and then use (\ref{detphi}) to obtain a convex function $\phi=\phi_{\frac{2}{n+1}}\in C(\overline{\Omega}) $ satisfying the Monge-Amp\`ere inequality
$$\det D^{2} \phi \geq [\diam\Omega]^{-\frac{2}{n+1}}~ \text{in}~ \Omega$$ while $\phi\leq 0$ on $\p\Omega$.
Hence, with $$C^*=(\lambda[\Omega]  [\diam\Omega]^{\frac{2}{n+1}})^{1/n}, $$ we have  $$\det D^2(C^*\phi)= [C^*]^n\det D^2\phi\geq \lambda[\Omega]\geq \lambda[\Omega] |u|^n
=\det D^2 u~ \text{in} ~\Omega,$$ while on $\p\Omega$, $u=0\geq C^*\phi$.
By the comparison principle, Lemma \ref{comp_prin}, we have
$u\geq C^*\phi_{\frac{2}{n+1}}$ in $\Omega.$
Therefore
 $$|u|\leq -C^*\phi_{\frac{2}{n+1}}~\text{in}~\Omega.$$
In particular, from the choice of $C_{\frac{2}{n+1}}$ in (\ref{alphaeq}), the bounds on $\lambda$ given by Corollary \ref{eigen_est}, we find that (\ref{ustep1}) now follows  the following
estimates
$$|u(x)|=|u(x', x_n)| \leq -C^*\phi_{\frac{2}{n+1}}(x', x_n) \leq C^*C_{\frac{2}{n+1}} x_n^{\frac{2}{n+1}}~\text{for all}~ x\in\Omega.$$

Next, we improve the estimate (\ref{ustep1}) by iteration. The key argument is the following.\\
{\bf Claim.} If for some $\beta\in (0, 1)$ we have
\begin{equation} |u(x)|\leq C(n,\beta, \diam \Omega) x_n^{\beta}~\text{for all}~x\in\Omega,
\label{ustep2}
\end{equation}
then for all $x\in\Omega$,
\begin{equation}|u(x)|\leq C(n,\alpha, \diam \Omega) x_n^{\alpha}~\text{for any}~
\beta<\alpha<  \min\{\beta+ \frac{2}{n}, 1\}.
\label{ustep3}
\end{equation}

From this claim, we can increase the exponent of $x_n$ in the upper bound for $|u|$ by at least $\frac{1}{n}$ if it is less than $1-\frac{1}{n}$. 
Thus, from (\ref{ustep1}), we obtain the desired bound (\ref{betapoly2}) after at most $(n-1)$ iterations. 

It remains to prove the claim.
Suppose we have (\ref{ustep2}) for $\beta<1$. If $\beta<\alpha<\min\{\beta +\frac{2}{n}, 1\} $, then
for $\hat C= \hat C(n, \beta, \alpha,\diam \Omega)$ large, we have 
\begin{equation}|u(x)|\leq C(n,\beta,\diam\Omega) x_n^{\beta} < \hat C\lambda^{-\frac{1}{n}} x_n^{\frac{n\alpha-2}{n}}~\text{in}~\Omega.
\label{Clarge}
\end{equation}

Denote by $(U^{ij})=(\det D^2 u) (D^2 u)^{-1}$ the cofactor matrix of the Hessian matrix $D^2 u = (u_{ij})$. Then $$\det U=(\det D^2 u)^{n-1} ~\text{and}~U^{ij} u_{ij}=n \det D^2 u= n\lambda |u|^n.$$ Using (\ref{detphi}), (\ref{Clarge}) and the matrix inequality
$$\trace (AB)\geq n (\det A)^{1/n} (\det B)^{1/n}~\text{for~} A, B~\text{symmetric}~\geq 0,$$ we find that 
\begin{eqnarray}U^{ij} (\hat C\phi_{\alpha})_{ij} \geq n\hat C (\det D^2 u)^{\frac{n-1}{n}} (\det D^2\phi_{\alpha})^{\frac{1}{n}}
 &\geq& n\hat C \lambda^{\frac{n-1}{n}}|u|^{n-1} x_n^{\frac{n\alpha-2}{n}}\nonumber \\&>& n\lambda |u|^n = n\det D^2 u=U^{ij} u_{ij}~\text{in}~\Omega.
 \label{incr_beta}
\end{eqnarray}
Now, the maximum principle for the operator $U^{ij}\p_{ij}$ applied to $u$ and $\hat C\phi_{\alpha}$ gives
$$u\geq \hat C\phi_{\alpha}~\text{in}~\Omega. $$
It follows that $$|u(x)|=|u|(x', x_n)\leq -\hat C\phi_{\alpha}(x', x_n) \leq \hat C C_{\alpha} x_n^{\alpha}~ \text{for all}~ x\in\Omega.$$
This gives (\ref{ustep3}) and the claim is proved. The proof of the proposition is complete.
\end{proof}
Inspecting (\ref{incr_beta}), we find that it also holds if we replace $u$ by a solution $v$ to (\ref{DMA}) where $n-2\leq p\neq n$ and $\alpha$ is required 
to satisfy 
$\beta<\alpha<\min\{\frac{\beta p + 2}{n},1
\}$.  This combined with the upper bound for $|v|$ in Lemma \ref{Alem}(iii) gives the almost Lipschitz property of $v$. We state this as a
proposition.
\begin{prop}
\label{pn2_prop}
Let $\Omega$ be an open bounded convex domain in $\R^n$ and $n-2\leq p\neq n$.
 Let $v\in C(\overline{\Omega})$ be a nonzero convex solution of 
 \begin{equation*}
\det D^{2} v=|v|^p~\text{in} ~\Omega,~\text{and}~
v =0~\text{on}~\p \Omega.
\end{equation*}
Then, 
 for all $\beta\in (0, 1)$, we have $v\in C^{0,\beta}(\overline{\Omega})$ and the estimate
 \begin{equation*}
 |v(x)|\leq C(n, p, \beta, \emph{diam} \Omega) [\emph{dist}(x,\p\Omega) ]^{\beta}~\text{for all}~x\in\Omega.
 \end{equation*}
\end{prop}
It would be interesting to see if the conclusions of Proposition \ref{pn2_prop} still hold for $p\in [0, n-2)$ when $n\geq 3$ and $\Omega$ is a 
general open bounded convex domain in $\R^n$. 
On the other hand, when $\Omega$ is a bounded domain with uniformly convex smooth boundary, nonzero convex solutions to (\ref{DMA}) share the same global smoothness property as the Monge-Amp\`ere eigenfunctions exactly when $p$ is a positive integer, regardless whether it is less than $n-2$ or not. More precisely, we have the following theorem.
\begin{thm}\label{Rlem} Let $\Omega$ be an open bounded domain with uniformly convex smooth boundary. Let $p>0$ and $p\neq n$. 
 Assume that $u\in C(\overline{\Omega})$ is a nonzero convex solution of 
 \begin{equation}
 \left\{
 \begin{alignedat}{2}
   \det D^{2} u~&= |u|^p \h~&&\text{in} ~\Omega, \\\
v &=0\h~&&\text{on}~\p \Omega.
 \end{alignedat}
 \right.
 \label{pnotn}
\end{equation}
\begin{myindentpar}{1cm}
 (i) If $p$ is a positive integer then $u\in C^{\infty}(\overline{\Omega}).$\\
 (ii) If  $p$ is a not an integer then $u\in C^{2 + [p],\beta}(\overline{\Omega}) $ for some $\beta=\beta (n, p)\in (0, 1)$ but $u\not\in C^{3+[p]}(\overline{\Omega}).$ Here we use $[p]$ to denote the greatest integer not exceeding $p$.
\end{myindentpar}
\end{thm}
In terms on the number of whole derivatives, the global regularity result in Theorem \ref{Rlem} is optimal. The proof of Theorem \ref{Rlem}, which is a revisit of 
the proof of the global smoothness of the Monge-Amp\`ere
 eigenfunctions in \cite[Theorem 1.4]{LS}, will be postponed to Section \ref{gl_sec}.

\subsection{Uniqueness of the Monge-Amp\`ere eigenvalue and eigenfunctions}
In this subsection, we prove the uniqueness property of the Monge-Amp\`ere eigenvalue and eigenfunctions (up to positive multiplicative constants)
to the eigenvalue problem (\ref{EVP_eq}).
\begin{prop}
\label{uni_lam}
Let $\Omega$ be an open bounded convex domain in $\R^n$.
 Suppose that $u, v\in C(\overline{\Omega})$ are nonzero convex functions in $\Omega$ with $u=v=0$ on $\p\Omega$. Suppose that they satisfy the equations
 \begin{equation*}
 \left\{
 \begin{alignedat}{2}
   \det D^{2} u~&=\lambda |u|^n \h~&&\text{in} ~\Omega, \\\
u &=0\h~&&\text{on}~\p \Omega,
 \end{alignedat}
 \right.
\end{equation*}
and 
\begin{equation*}
 \left\{
 \begin{alignedat}{2}
   \det D^{2} v~&=\Lambda |v|^n \h~&&\text{in} ~\Omega, \\\
v &=0\h~&&\text{on}~\p \Omega
 \end{alignedat}
 \right.
\end{equation*}
respectively, for some positive constants $\lambda$ and $\Lambda$. Then $\lambda=\Lambda$ and $v=m u$ for some positive constant $m$.
\end{prop}
\begin{proof} 
By Proposition \ref{reg_prop}, we have $u, v\in C^{\infty}(\Omega)$.\\
{\it Step 1 (Uniqueness of the eigenvalue: $\lambda=\Lambda$).} We apply Proposition \ref{detC} to $u$ and $v$ to obtain
 $$\int_{\Omega}  \Lambda |u||v|^n~dx=\int_{\Omega} |u|\det D^2 v~dx\geq \int_{\Omega}(\det D^2 u)^{\frac{1}{n}}(\det D^2 v)^{\frac{n-1}{n}}|v|~dx= 
 \int_{\Omega} \lambda^{\frac{1}{n}}\Lambda ^{\frac{n-1}{n}}|u||v|^n~dx.$$
 Since $|u||v|^n>0$ in $\Omega$, it follows that $\Lambda \geq \lambda.$ Similarly, applying Proposition \ref{detC} to $v$ and $u$, we obtain $\lambda\geq\Lambda$. Therefore, $\lambda=\Lambda$.\\
 {\it Step 2 (Uniqueness of the eigenfunctions up to positive multiplicative constants: $v=m u$).} In this step, we can use $$\Lambda=\lambda=\lambda[\Omega]$$ as established in {\it Step 1}.
 We can normalize both $u$ and $v$ by multiplying suitable constants so that 
$$\|u\|_{L^{\infty}(\Omega)}= \|v\|_{L^{\infty}(\Omega)}= 1. $$
We need to show that $u=v$. 
Crucial to our argument is the following integrability lemma which is the normalized version of Lemma \ref{inte_lem0} stated in the Introduction.
\begin{lem}
\label{inte_lem}
Let $\Omega$ be an open bounded convex domain in $\R^n$.
Suppose that $u, v\in C(\overline{\Omega})\cap C^{\infty}(\Omega)$ are nonzero convex eigenfunctions to (\ref{EVP_eq}) with
$\|u\|_{L^{\infty}(\Omega)}= \|v\|_{L^{\infty}(\Omega)}= 1. $
Then, there exists a constant $C=C(n,\Omega)$ depending only $n$ and $\Omega$ so that 
\begin{equation}\int_{\Omega} \Delta u |v|^{n-1}~dx\leq C
\label{Dest20}
\end{equation}
and
\begin{equation}|Du(x)| |v(x)|\leq C(n,\emph{diam} \Omega)  \emph{dist}^{\frac{n-1}{n+1}}(x,\p\Omega).
\label{Dest10}
\end{equation}

\end{lem}
We will prove Lemma \ref{inte_lem} at the end of this section and now continue with our proof of {\it Step 2}.
Because $u+ v$ is smooth and convex in $\Omega$, by the Arithmetic-Geometric inequality, we have
$$n(\det D^2 (u+ v))^{\frac{1}{n}}\leq \Delta (u+ v).$$
From Lemma \ref{inte_lem} and Corollary \ref{eigen_est}, we find that
\begin{equation}\int_{\Omega}(\det D^2 (u+ v))^{\frac{1}{n}} (\det D^2 v)^{\frac{n-1}{n}} ~dx\leq \int_{\Omega}\lambda^{\frac{n-1}{n}}\Delta (u+ v)|v|^{n-1}~dx\leq C(n,\Omega).
\label{inte_check}
\end{equation}
On the other hand, using the matrix inequality
$$[\det (A+ B)]^{\frac{1}{n}}\geq (\det A)^{\frac{1}{n}} + (\det B)^{\frac{1}{n}} ~\text{for}~ A, B~\text{symmetric, positive definite}$$
with equality if and only if $A= cB$ for some positive constant $c$, we find that for all $x\in\Omega$
\begin{equation}
\label{lam_uv}
(\det D^2 (u+ v)(x))^{\frac{1}{n}} \geq (\det D^2 u(x))^{\frac{1}{n}}  + (\det D^2  v(x))^{\frac{1}{n}} =  \lambda^{\frac{1}{n}} |u(x) + v(x)|
\end{equation}
with equality if and only if $D^2 u(x)= C(x)D^2 v(x)$ for some positive constant $C(x)$.

Due to (\ref{inte_check}), and the inequality $$\int_{\Omega}\det D^2 v~dx =\int_{\Omega}\lambda |v|^n~dx\leq \lambda|\Omega|\leq C(n,\Omega),$$ we can apply Proposition \ref{detC} to $u+ v$ and $v$ and use (\ref{lam_uv}) to obtain
 \begin{eqnarray*}\int_{\Omega} |u+ v| \lambda |v|^n~dx= \int_{\Omega}|u+v|\det D^2 v ~dx&\geq&\int_{\Omega}(\det D^2 (u+ v))^{\frac{1}{n}} (\det D^2 v)^{\frac{n-1}{n}} |v|~dx\\ &\geq& \int_{\Omega} \lambda^{\frac{1}{n}}|u+ v| \lambda ^{\frac{n-1}{n}}|v|^n~dx.
 \end{eqnarray*}
 It follows that we must have equality in (\ref{lam_uv}) for all $x\in\Omega$. Hence, for all $x\in\Omega$, 
 $$D^2 u(x) =C(x) D^2 v(x).$$  Taking the determinant of both sides, and using (\ref{EVP_eq}), we find $$C(x)= \frac{|u(x)|}{|v(x)|}= \frac{u(x)}{v(x)}$$ and therefore
 \begin{equation}D^2 u(x) = \frac{u(x)}{v(x)} D^2 v(x).
 \label{uvD2}
 \end{equation}
 Using (\ref{uvD2}) and (\ref{Dest10}), we conclude the proof of {\it Step 1} as follows. 
 
 Without loss of generality, we can assume that $\Omega$ contains the origin in its interior. For any direction $\eb\in S^{n-1}=\{x\in \R^n: |x|=1\}$, the ray through the origin in the $\eb$ direction intersects the boundary $\p\Omega$
 at $x_0$ and $x_1$. Assume that the segment from $x_0$ to $x_1$ is given by $x_0 + t \eb$ for $0\leq t \leq m:=|x_0 x_1|$.
 Let us consider the following non-positive single variable functions on $[0, m]$:
 $$f(t)= u(x_0 + t \eb),~\tilde f(t) = v(x_0 + t \eb).$$
 Then we have 
 $$f(0)=f(m)=\tilde f(0)=\tilde f(m)=0, f''(t) = D^2 u(x_0 + t\eb) \eb\cdot \eb, \tilde f''(t) = D^2 v(x_0 + t\eb) \eb\cdot \eb.$$
 Using (\ref{uvD2}), we find
 $$f''(t) =\frac{f(t)}{\tilde f'(t)}\tilde f''(t)$$
 for all $t\in (0, m)$. Thus, $(f'(t) \tilde f(t)-f(t)\tilde f'(t))'=0$ for all $t\in (0, m)$ and hence $$f'(t) \tilde f(t)-f(t)\tilde f'(t)= C$$ on $(0, m)$ for some constant $C$.

Using (\ref{Dest10}) and recalling $n\geq 2$, we find that when $t\rightarrow 0$, $f'(t) \tilde f(t)\rightarrow 0$ and $f(t)\tilde f'(t)\rightarrow 0.$ Thus $C=0$. Therefore 
 $(f(t)/\tilde f(t))'=0$ on $(0, m)$. It follows that $f(t) = c \tilde f(t)$ on $(0, m)$ for some constant $c>0$. 
 
 Returning to $u$ and $v$, we find that the ratio $u/v$ is a constant $c(\eb)$ in the direction $\eb\in S^{n-1}$. We vary the direction $\eb$ and use the fact that $c(\eb)=\frac{u(0)}{v(0)}$ for all $\eb\in S^{n-1}$ to conclude that $u/v$ is a positive constant $c$ in $\Omega$. Since $\|u\|_{L^{\infty}(\Omega)}= \|v\|_{L^{\infty}(\Omega)}= 1$, we have
 $c=1$. Therefore $u=v$.
\end{proof}
\begin{proof}[Proof of Lemma \ref{inte_lem}]
For $x\in\Omega$, by the convexity of $u$ and the fact that $u=0$ on $\p\Omega$, we have the gradient estimate
\begin{equation}|Du(x)|\leq \frac{|u(x)|}{\dist (x, \p\Omega)}.
\label{Du_est}
\end{equation}
Let $\{\Omega_{m}\}\subset\Omega$ be a sequence of smooth, open, bounded and uniformly convex domains that converges to $\Omega$ in the Hausdorff distance.  
To prove (\ref{Dest20}), by the monotone 
convergence theorem, it suffices to prove that for all $m$, we have
\begin{equation*}
\int_{\Omega_m} \Delta u |v|^{n-1}~dx\leq C(n, \Omega).
\end{equation*}
For each $m$, using the interior smoothness of $u$ and $v$ and integrating by parts, we have
\begin{equation}
\label{IBP}
\int_{\Omega_m} \Delta u |v|^{n-1}~dx =\int_{\Omega_m} \Delta u (-v)^{n-1}~dx=\int_{\Omega_m} (n-1)Du\cdot Dv |v|^{n-2} ~dx +\int_{\p\Omega_m} \frac{\p u}{\p \nu}|v|^{n-1}.
\end{equation}
where $\nu$ is the outer normal unit vector field on $\p\Omega_m$. 

We show each term on the right hand side of (\ref{IBP}) is uniformly bounded. Let us choose $\beta=\frac{n}{n+1}\in (0, 1)$ in Proposition \ref{alm_Lip}. 

Then,
for $x\in\Omega$, by Proposition \ref{alm_Lip} and the gradient estimate (\ref{Du_est}), we have
\begin{equation}|Du(x)| |v(x)|\leq C(n,\beta, \diam \Omega)  \dist^{2\beta-1}(x,\p\Omega)\leq  C(n,\diam \Omega) \dist^{\frac{n-1}{n+1}}(x,\p\Omega)
\label{Dest1}
\end{equation}
and
\begin{equation}|Du(x)| |Dv(x)| |v(x)|^{n-2} \leq C(n, \diam \Omega)  \dist^{n\beta-2}(x,\p\Omega)\leq  C(n, \diam \Omega)  \dist^{-\frac{2}{n+1}}(x,\p\Omega).
\label{Dest2}
\end{equation}
As a consequence of (\ref{Dest1}), we have (\ref{Dest10}) and furthermore, for all $x\in\Omega$,
$$|Du(x)| |v(x)|^{n-1} \leq |Du(x)| |v(x)| \leq C(n,\diam\Omega).$$
Hence,
we can easily see that for all $m$,
$$\int_{\p\Omega_m} \frac{\p u}{\p \nu}|v|^{n-1} \leq C(n, \diam \Omega).$$
It remains to show the uniform boundedness of the first term on the right hand side of (\ref{IBP}). However, this is easy, since by (\ref{Dest2}), and the fact that $\Omega_m\subset\Omega$, we have
$$\int_{\Omega_m} Du\cdot Dv |v|^{n-2} ~dx\leq \int_{\Omega} |Du| |Dv| |v|^{n-2}~dx \leq \int_{\Omega} C(n, \diam \Omega)  \dist^{-\frac{2}{n+1}}(x,\p\Omega)~dx \leq C(n,\Omega).$$
\end{proof}

\subsection{Stability of the Monge-Amp\`ere eigenvalue} In this subsection, we prove further properties of the Monge-Amp\`ere eigenvalue defined by (\ref{lam_def}). These include its stability with respect to the Hausdorff convergence of the domains.
\begin{prop} 
\label{lam_nice}
The Monge-Amp\`ere eigenvalue, as defined by (\ref{lam_def}), has the following properties:
\begin{myindentpar}{1cm}
(i) \emph{(Stability with respect to the Hausdorff convergence)} Let $\{\Omega_m\}_{m=1}^{\infty}\subset\R^n$ be a sequence of open bounded convex domains that 
converges to an open bounded convex domain $\Omega$. Then $\lim_{m\rightarrow \infty}\lambda[\Omega_m]=\lambda[\Omega]$.\\
(ii) \emph{(Domain monotonicity)} If $\Omega_1$ and $\Omega_2$ are open bounded convex domains such that $\Omega_1\subset\Omega_2$ then $\lambda[\Omega_2]\leq \lambda[\Omega_1].$\\
 (iii) \emph{(Affine transformation)} Let $\Omega$ be an open bounded convex domain in $\R^n$. If $T:\R^n\rightarrow \R^n$ is an affine transformation on $\R^n$
 then $\lambda[T(\Omega)]=|\det T|^{-2}\lambda [\Omega]$. In particular, 
 \begin{myindentpar}{1.2cm}
 $\bullet$ \emph{(Homogeneity with respect to domain dilations)} If $t>0$, then $\lambda [t\Omega]= t^{-2n} \lambda [\Omega].$\\
 $\bullet$ \emph{(Affine invariance)} If $T:\R^n\rightarrow \R^n$ is an affine transformation on $\R^n$
 with $\det T=1$, then $\lambda[T(\Omega)]=\lambda [\Omega]$.
 \end{myindentpar}
 \end{myindentpar}
\end{prop}

\begin{proof}
(i) We first consider
the case $\{\Omega_m\}_{m=1}^{\infty}$ is a sequence of open, smooth, bounded and uniformly convex domains. In this case, 
the stability property of the Monge-Amp\`ere eigenvalue, that is,
$$\lim_{m\rightarrow \infty}\lambda[\Omega_m]=\lim_{m\rightarrow \infty}\lambda(\Omega_m)=\lambda[\Omega]$$
is a consequence of Proposition \ref{Lam_prop} and the uniqueness of the Monge-Amp\`ere eigenvalue as proved in Proposition \ref{uni_lam}. 

For the general open bounded convex domains $\Omega_m$, we approximate each $\Omega_m$ by a sequence of smooth, open, bounded and uniformly convex domains $\{\Omega_{km}\}_{k=1}^{\infty}$ that converges to $\Omega_m$ in the Hausdorff distance. Then, as above, 
$$\lim_{k\rightarrow \infty}\lambda(\Omega_{km})=\lambda[\Omega_m].$$
Now, the conclusion of (i) follows from a standard diagonal argument.
\\
(ii) Suppose that $\Omega_1\subset\Omega_2$. Let $\{\Omega_{1m}\}\subset\Omega_1$ be a sequence of open, smooth, bounded and uniformly convex domains that converges to $\Omega_1$ in the Hausdorff distance. 
Let $\{\Omega_{2m}\}\supset\Omega_2$ be a sequence of open, smooth, bounded and uniformly convex domains that converges to $\Omega_2$ in the Hausdorff distance. Then $\Omega_{1m}\subset\Omega_{2m}$ for all $m$. 
By \cite[Theorem 4.1(iii)]{W}, we have $\lambda(\Omega_{2m})\leq \lambda (\Omega_{1m})$ for all $m$. Letting $m\rightarrow\infty$ and using (i), we find
$\lambda[\Omega_2]\leq \lambda[\Omega_1]$.\\
(iii) Let $u\in C(\overline{\Omega})$ with $u=0$ on $\p\Omega$ be a nonzero convex eigenfunction of the Monge-Amp\`ere operator on $\Omega$ as given by Proposition \ref{exi_prop}. Then
 $\det D^2 u= \lambda |u|^n$ in $\Omega$ where $\lambda=\lambda[\Omega]$. Suppose $T:\R^n\rightarrow \R^n$ is an affine transformation on $\R^n$. Then
$$u_T(x):= u(T^{-1}x), x\in T(\Omega)$$
satisfies in $T(\Omega)$
$$D^2 u_T(x) = (T^{-1})^t D^2 u(T^{-1}x) T^{-1}$$
and hence
$$\det D^2 u_T(x) = |\det T|^{-2} (\det D^2 u) (T^{-1}x)= |\det T|^{-2} \lambda |u|^n(T^{-1}x)= \lambda |\det T|^{-2} |u_T(x)|^n.$$
Therefore, from the uniqueness of the Monge-Amp\`ere eigenvalue, we have
$$\lambda [T(\Omega)]= |\det T|^{-2} \lambda [\Omega].$$ 
This proves (iii).

Finally, when applying (iii) to $Tx = tx$ with $\det T= t^n$ ($t>0$), we obtain the homogeneity property with respect to domain dilations of the Monge-Amp\`ere eigenvalue.
\end{proof}

\section{Proofs of Theorems \ref{lamBM} 
and \ref{extreme_thm}}
\label{iso_sec}
In this section, we prove Theorems \ref{lamBM} 
and \ref{extreme_thm}.
\begin{proof}[Proof of Theorem \ref{BMcor}] Let $\Omega_0$ and $\Omega_1$ be open bounded convex domains in $\R^n$. For each $i=0, 1$, let $\{\Omega_{i m}\}_{m=1}^{\infty}$ be a 
sequence of open, bounded, smooth, and uniformly convex domain in $\R^n$ converging to $\Omega_i$ in the Hausdorff 
distance. We recall the Minkowski linear combinations of two convex domains as defined by (\ref{Min_lin}). Then for each $\alpha\in [0, 1]$, the sequence of Minkowski 
linear combinations $\{\Omega_{\alpha m}\}$ of $\Omega_{0m}$ and $\Omega_{1m}$ converges to $\Omega_{\alpha}$ in the Hausdorff distance.
For each $m$, we have by (\ref{lamBM}),
$$ \lambda(\Omega_{\alpha m})^{-\frac{1}{2n}}\geq (1-\alpha) \lambda (\Omega_{0m})^{-\frac{1}{2n}} + \alpha \lambda(\Omega_{1m})^{-\frac{1}{2n}}.$$
Letting $m\rightarrow\infty$ in the above inequality and using the stability with respect to the Hausdorff convergence of the Monge-Amp\`ere eigenvalue in Theorem \ref{lam_thm} (iv), we obtain 
$$ \lambda[\Omega_{\alpha }]^{-\frac{1}{2n}}\geq (1-\alpha) \lambda [\Omega_{0}]^{-\frac{1}{2n}} + \alpha \lambda[\Omega_{1}]^{-\frac{1}{2n}}.$$
The theorem is proved.
\end{proof}
\begin{proof}[Proof of Theorem \ref{extreme_thm}]
(i) Let $\{\Omega_m\}$ be a sequence of open, bounded, smooth, and uniformly convex domains in $\R^n$ converging to $\Omega$ in the Hausdorff distance. For 
each $m$, let $\mathcal{B}(\Omega_m)= B_{r_m}$ be the (open) ball
having the same volume as $\Omega_m$. Then, by Theorem 3.2 in Brandolini-Nitsch-Trombetti \cite{BNT}, we have 
\begin{equation}\lambda (\Omega_m)\leq\lambda (\mathcal{B}(\Omega_m))=\lambda (B_{r_m}).
 \label{bnt_ineq}
\end{equation}
Since $|\Omega_m|= |B_{r_m}|=  r_m^n |B_1|$ and, 
 $\lambda (B_{r_m})= \frac{\lambda (B_1)}{r_m^{2n}}$ by Proposition \ref{lam_nice}, we have
\begin{equation}\lambda (\mathcal{B}(\Omega_m))= \frac{\lambda (B_1)}{r_m^{2n}}= \frac{|B_1|^2 \lambda (B_1)}{|\Omega_m|^2}.
 \label{lam_rm}
\end{equation}
Similarly, if $\mathcal{B}(\Omega)$ is the open ball having the same volume as $\Omega$ then 
\begin{equation}\lambda[\mathcal{B}(\Omega)]=\lambda (\mathcal{B}(\Omega)= \frac{|B_1|^2 \lambda (B_1)}{|\Omega|^2}.
 \label{lam_r}
\end{equation}
Letting $m\rightarrow\infty$ in (\ref{bnt_ineq}), using Theorem \ref{lam_thm} (iv) on the stability of the Monge-Amp\`ere eigenvalue with respect to the Hausdorff convergence of the domains, and 
recalling (\ref{lam_rm}) and (\ref{lam_r}), we obtain (\ref{iso_ineq}), that is,
 $$\lambda[\Omega]\leq\lambda [\mathcal{B}(\Omega)] =\frac{|B_1|^2 \lambda (B_1)}{|\Omega|^2}.$$
(ii) Given a fixed volume $V>0$, let
\begin{equation*}\displaystyle \lambda_{V}=\inf\left\{\lambda[\Omega]: \Omega\subset\R^n~ \text{is open, bounded, convex with } |\Omega|=V\right\}.
\end{equation*}
Let  $\{\Omega_m\}_{m=1}^{\infty}$ be a sequence of open bounded convex sets in $\R^n$ having volume $V$ such that
\begin{equation}\label{min_seq}\lambda[\Omega_m]\rightarrow \lambda_V.
\end{equation}
For each $m$, 
 by John's lemma, Lemma \ref{J_lem}, there exists an affine transformation $T_m$ of $\R^n$ with $\det T_m=1$ such that
 \begin{equation}B_{R_m}\subset T_m(\Omega_m)\subset B_{nR_m}
 \label{OR_eq}
 \end{equation}
 for some $R_m>0$. We find
 \begin{equation}R_m^n \omega_n\leq |\Omega_m|=V\leq n^n \omega_n R_m^n,~\omega_n=|B_1|.
  \label{vol_R}
 \end{equation}
From (\ref{OR_eq}), (\ref{vol_R}) and the Blaschke Selection Theorem (see, for example, \cite[Theorem 1.8.7]{Sch}), there exists a subsequence of $\{T_m(\Omega_m)\}$, still 
denoted by $\{T_m(\Omega_m)\}$, that converges in the Hausdorff distance to an open bounded convex set
$S$ with volume $V$. By the stability of the Monge-Amp\`ere eigenvalue in Theorem \ref{lam_thm}(iv), we have
$$\lambda[T_m(\Omega_m)]\rightarrow \lambda[S].$$
Since $\lambda[T_m(\Omega_m)]=\lambda[\Omega_m]$ by Proposition \ref{lam_nice}, we find from (\ref{min_seq}) that
$\lambda[S]=\lambda_V$. Thus, the open bounded convex set $S$ has the smallest Monge-Amp\`ere eigenvalue among all open bounded convex sets $\Omega$ having the same volume $V$.\\
(iii) The proof is similar to that of (ii) noting that if $\Omega_m$ in (ii) is taken to be centrally symmetric then $T_m(\Omega_m)$ is also centrally symmetric
and the limit $K$ of $T_m(\Omega)$ in the Hausdorff metric is also centrally symmetric. 
\end{proof}

\section{Proof of Theorem \ref{Rlem}}
\label{gl_sec}
\begin{proof}[Proof of Theorem \ref{Rlem}] The proof of the theorem is divided into several steps. We note that, by Proposition \ref{reg_prop}, $u\in C^{\infty}(\Omega)$.

{\it Step 1 ($u\not\in C^{3+[p]}(\overline{\Omega})$ if $p$ is not an integer).} We argue by contradiction that $u\not\in C^{3+[p]}(\overline{\Omega})$ if $p$ is not an integer. Suppose otherwise, then we can write for $x\in\Omega$
 $$|u(x)|=-u(x)= g(x) \dist(x,\p\Omega)$$
 where the function $g$ is positive (due to (\ref{ulow}) and (\ref{uup}) below) and  $C^{2+[p]}$ near the boundary $\p\Omega$. Note that $\dist(\cdot,
 \Omega)$ is smooth near the boundary $\p\Omega$. By using the hypothetical assumption  $u\in C^{3+[p]}(\overline{\Omega})$
 and the equation
 $$\det D^2 u= |u|^p = g^p \dist^p(x,\p\Omega),$$
 we find that the right hand side $ g^p \dist^p(x,\p\Omega)$ is $C^{1+[p]}$ near the boundary $\p\Omega$. This is impossible because the function
 $\dist^p(x,\p\Omega)$ is not $C^{1+[p]}$ near the boundary $\p\Omega$.

 We now return to proving higher order regularity results which follow from revisiting the proof of the global smoothness of
 the Monge-Amp\`ere
 eigenfunctions in \cite[Theorem 1.4]{LS}.
 For reader's convenience, we sketch the arguments here.
 
{\it Step 2 ($u\in C^{2,\beta} (\overline{\Omega})$ for all $\beta < \min \{p, \frac{2}{2+p}\}$).}
Let $M=\|u\|_{L^{\infty}(\Omega)}\in (0,\infty)$.
As in the proof of Lemma \ref{Alem}, we obtain from the convexity of $u$
that
\begin{equation}
 \label{ulow}
 |u(x)|\geq \frac{M \dist(x,\p\Omega)}{\diam \Omega}~\text{for all}~x\in\Omega.
\end{equation}
Now, we prove that
\begin{equation}
 \label{uup}
 |u(x)|\leq C(n, p,\Omega, M) \dist(x,\p\Omega)~\text{for all}~x\in\Omega.
\end{equation}
Let $\rho$ be a $C^2$, strictly convex defining function of $\Omega$, that is 
$\Omega:=\{x\in \R^n: \rho(x)<0\}$, $\rho=0$ on $\p\Omega$ and $D\rho\neq 0$ on $\p\Omega$.
Then $$D^2 \rho\geq \eta I_n~ \text{and} ~\rho\geq -\eta^{-1} ~\text{in}~ \Omega$$ for some $\eta>0$ depending only on $\Omega$. 
Consider the following function
$\tilde u(x) =\mu (e^{\rho}-1).$
From
$$D^2 (e^{\rho}-1)=e^{\rho}(D^2\rho + D\rho \otimes D\rho)\geq e^{-\eta^{-1}}\eta I_n,$$
we find that $\tilde u$  is convex and furthermore, for a sufficiently large $\mu=\mu(n, p, M,\Omega)$, $$\det D^2 \tilde u\geq |u|^p=\det D^2 u~\text{in}~\Omega.$$
Note that $\tilde u= u=0$ on $\p\Omega$. Thus, by the comparison principle, Lemma \ref{comp_prin}, we have $u\geq \tilde u$ in $\Omega$. Hence $|u|\leq |\tilde u|$ and (\ref{uup})
follows. 

The inequalities (\ref{ulow}) and (\ref{uup}) imply that 
if $x_0 \in \p \Omega$ then $c \leq |D u(x_0)| \le C$. As a consequence, using the smoothness and uniform convexity of $\p\Omega$, we find that
on $\p \Omega$ the function $u$ separates quadratically from its tangent plane at each $x_0\in\p\Omega$, that is,
$$\rho|x-x_0|^2\leq u(x)-u(x_0)-Du(x_0)\cdot (x-x_0)\leq \rho^{-1}|x-x_0|^2~\forall x\in\p\Omega$$
for some positive constant $\rho=\rho(n, p, M, \Omega)$. With this, we can follow the proof of Theorem 1.3 in Savin \cite{S} to obtain the global
$C^{2}(\overline{\Omega})$ regularity of $u$. Now, write $|u|= g \dist(\cdot, \p\Omega)$ with $g\in C^{0, 1}(\overline{\Omega})$. Then the conditions
of Theorem 1.2 in \cite{LS} are satisfies and therefore, we can conclude from this theorem that 
that $u\in C^{2,\beta} (\overline{\Omega})$ for all $\beta < \min \{p, \frac{2}{2+p}\}$.

 {\it Step 3 (Transform equation (\ref{pnotn}) into an equivalent equation in the upper half space).} Since $u\in C^{\infty}(\Omega)$, it remains to investigate the smoothness of $u$ near the boundary $\p\Omega$.
Assume that $0\in\partial\Omega$ and 
$e_n=(0, \cdots, 0, 1)\in\R^n$ is the inner normal of $\partial\Omega$ at $0$. We make the rotation of coordinates
$$y_{n}=- x_{n+ 1},~ y_{n+ 1}= x_n,~ y_k= x_k~(1\leq k\leq n-1).$$
In the new coordinates, the graph of $u$ near the origin can be represented as 
$y_{n+1} = \tilde u(y)$ in the upper half-space $\R^{n}_{+}=\{y\in \R^n: y_{n}>0\}$.
Near the origin, the boundary $\p\Omega$ is given by $x_n = \phi(x')$ in the original coordinates. Thus, the boundary condition for $\tilde u$ is 
$\tilde u=\phi$ on $\{y_n=0\}.$ Computing using the Gauss curvature as in the derivation of equation (5.2) in \cite{LS}, we find that, locally, for some small $r_0>0$
(now relabeling $y$ by $x$), $\tilde u_n >c$ in $B_{r_0}^{+}$, and 
\begin{equation}\label{ET}
\left\{\begin{array}{rl}
\det D^2 \tilde u  &=  x_n^p \, \, \tilde u_n^{n+2} \quad \mbox{in}~ B_{r_0}^{+},\\
\tilde u&= \phi \quad \quad \quad \quad \mbox{on}~\{x_n=0\}\cap B_{r_0}.
\end{array}\right.
\end{equation}
In (\ref{ET}) and in what follows, we use the notation
$$B_r^{+}:=B_{r}(0)\cap \{x=(x_1,\cdots, x_n)\in\R^n: x_n>0\}.$$
Since we have already prove $u \in C^{2,\beta}(\overline {\Omega})$, we have $\tilde u \in C^{2,\beta}(\overline {B_{r_0}^+}) $ for some small $\beta>0$, 
and $\tilde u_n >c$. It remains to study the higher order regularity of solutions to 
\eqref{ET} in a neighborhood of the origin.  
For simplicity of notation, we relabel $\tilde u$ from (\ref{ET}) by $u$.

Next we perform the following partial
Legendre transformation to the solutions of \eqref{ET}:
\begin{equation}\label{leg}
y_i= u_i(x) \quad (i\le n-1),\quad  y_n = x_n,\qquad u^*(y)=x' \cdot \nabla_{x'}u -u(x).
\end{equation}
The function $u^*$ is obtained by taking the Legendre transform of
$u$ on each slice $x_n=\text{constant}.$ 
Note that $(u^*)^* =u.$
As in \cite[equation (5.4)]{LS}, we find that if $u$ satisfies (\ref{ET}) then $u^*$ (which is convex
in $y'$ and concave in $y_n$) satisfies
\begin{equation}\label{eqn-u*}
\left\{\begin{array}{rl}
y_n^{\alpha} (-u^*_n)^{n+ 2} \det D^2_{y'}u^* + u^*_{nn}  &= 0~\mbox{in}~ B_\delta^{+}\\
u^*&= \phi^*~~~~\mbox{on}~\{y_n=0\}\cap B_\delta,
\end{array}\right.
\end{equation}
where $\alpha=p$. Moreover $u^* \in C^{2,\beta}(\ov B_\delta^+)$, $-u_n^*>c$ and $\phi^* \in C^\infty$.

 {\it Step 4 (Regularity of $u^*$ in the non-degenerate directions $y'$).} In order to obtain the smoothness of $u^*$ from \eqref{eqn-u*} we establish Schauder estimates for its linearized equation. We consider linear equations of the form
\begin{equation}x_n^{\alpha} \sum_{i, j\leq n-1}a^{ij} v_{ij} + v_{nn}= x_n^{\alpha} f(x)
 \label{veq}
\end{equation}
with $a^{ij}$ uniformly elliptic,
that is,
$$\lambda |\xi|^2\leq \sum_{i, j=1}^{n-1} a^{ij}(x)\xi_i\xi_j \leq \Lambda |\xi|^2~\text{for some positive constants}~\lambda~\text{and}~\Lambda~\text{and for all}~\xi\in\R^{n-1}.$$
\begin{defn}\label{dal}
We define the distance $d_\alpha$ between 2 points $y$ and $z$ in the upper half-space by
$$d_\alpha(y,z):=|y'-z'| +\left|y_n^\frac{2+\alpha}{2}-  z_n^\frac{2+\alpha}{2}\right|.$$
\end{defn}
The relation between $d_\alpha$ and the Euclidean distance in $\overline{B}_1^+$ is as follows:
\begin{equation}\label{51}
c|y-z|^\frac{2+\alpha}{2} \le d_\alpha(y,z) \le C |y-z| ,
\end{equation}
$$ d_\alpha(y,z) \sim |y-z|  \quad \mbox{if} \quad y,z \in \overline{B}_1^+\cap \{x_n \ge 1/4  \}.$$

If function $w$ is $C^\gamma$ with respect to $d_\alpha$ (with $\gamma \in (0, \frac{2}{2+\alpha})$) we write
$w \in C_\alpha^\gamma(\ov B_1^+).$ 
In view of \eqref{51} we obtain the following relations between the $C^\gamma_\alpha$ spaces and the standard $C^\gamma=C^{0,\gamma}$ H\"older spaces: 
\begin{align*}
 w \in C^\gamma_\alpha(\ov B_1^+)  \quad & \Rightarrow \quad w \in C^\gamma (\ov B_1^+) \\
 w \in C^\beta (\ov B_1^+)  \quad &  \Rightarrow \quad w \in C^\gamma_\alpha(\ov B_1^+) \quad \mbox{with} \quad \gamma=\beta \, \,  \frac{2}{2+\alpha}.
\end{align*}
We have the following Schauder estimates; see \cite[Proposition 5.3]{LS}.
\begin{prop}
 \label{LMA-model1}
 Assume that $v$ solves \eqref{veq} in $\ov B_\delta^+$ and
$v= \varphi(x')~\text{on}~ \{x_n=0\}\cap \ov B_\delta^+.$
If $a^{ij}$, $f \in C_\alpha^\gamma(\ov B_\delta^+)$ with $\frac{\gamma}{2} \le \frac{\min\{1,\alpha\}}{2+\alpha}$, and $\varphi \in C^{2,\gamma}$, then
$D v, D^2v \in C^\gamma_\alpha(\ov B_{\delta/2}^+).$
\end{prop}

By repeatedly differentiating (\ref{veq}) in the $x'$ direction we easily obtain Schauder estimates for higher derivatives. Below $m=(m_1,..,m_{n-1})$ denotes a multi-index with $m_i$ nonnegative integers.
\begin{cor}\label{c0}
If in the proposition above $\varphi \in C^{k+2,\gamma}$ for some integer $k \ge 0$ and
$$D^m_{x'}a^{ij}, D^m_{x'}f \in C^\gamma_\alpha (\ov B_\delta^+) \quad \forall \, \, m \quad \mbox{with} \quad |m| \le k,$$
then
$$D D^m_{x'} v, D^2 D^m_{x'}v \in C^\gamma_\alpha(\ov B_{\delta/2}^+) \quad   \quad \forall \, \,m \quad \mbox{with} \quad   |m| \le k.$$
\end{cor}

As mentioned right after equation (\ref{ET}), it suffices to study higher order regularity in a neighborhood of the origin for the function $u^*$ satisfying \eqref{eqn-u*}
with $\alpha=p$, $u^* \in C^{2,\beta}(\ov B_\delta^+)$, $-u_n^*>c$ and $\phi^* \in C^\infty$.

Fix $k<n$. Then 
$v:= u^*_k=\frac{\p u^*}{\p y_k}$ solves the linearized equation
\begin{equation}
 \label{LMA_uk}
y_n^{\alpha} \sum_{i, j\leq n-1}a^{ij}  v_{ij} + v_{nn}= y_n^{\alpha} f(y)\quad\mbox{in}\quad B_\delta^+
\end{equation}
where
$$a^{ij} =(-u^*_n)^{n+ 2} U_{y'}^{\ast ij}, \quad \quad \quad f(y) = (n+2)(-u_n^{\ast})^{n+1} u^*_{nk}\det D^2_{y'}u^*.$$
and $U^*_{y'}$ denotes the cofactor matrix of $D^2_{y'}u^*$.

Since $u^* \in C^{2,\beta}(\ov B_\delta^+)$ we obtain $Du^*$, $D^2u^* \in C^\gamma_\alpha(\ov B_\delta^+)$ for some small $\gamma>0$, hence $a^{ij}$, $f \in C^\gamma_\alpha(\ov B_\delta^+)$.

By Proposition \ref{LMA-model1},
$D^2v \in C^\gamma_\alpha$ up to the boundary in $\ov B_{\delta/2}^+ $ which in turn implies $D_{y'}a^{ij}$, $D_{y'}f \in C^\gamma_\alpha (\ov B_{\delta/2}^+)$. Now 
we may apply Corollary \ref{c0} and iterate this argument to obtain that  $D^m_{y'}D^l_{y_n}u^*$ with $l\in \{0,1,2\}$ are continuous up to the boundary 
in $\ov B_{\delta/2}^+ $ for all multi-indices $m \ge 0$. 

 {\it Step 5 (Regularity of $u^*$ in the degenerate directions $y_n$ and completion of the proof).}
In order to obtain the continuity of these derivatives for larger values of $l$ we differentiate the 
equation (\ref{eqn-u*}) with respect to $y_n$.
\begin{myindentpar}{1cm}
{\it Case 1: $\alpha=p$ is a positive integer.} Then we can repeatedly differentiate (\ref{eqn-u*}) with respect to $y_n$ and each derivative $D^m_{y'}D^l_{y_n}u^*$ with $l \ge 3$ can be expressed as a polynomial 
involving nonnegative powers of $y_n$ and derivatives $D_{y'}^q D_{y_n}^su^*$ with $s< l$. Thus $u^* \in C^\infty (\ov B_{\delta/2}^+)$ by an induction argument. It follows that 
$u\in C^{\infty}(\overline{\Omega})$ as desired. \\
{\it Case 2: $\alpha=p$ is not a positive integer.} Then we can differentiate (\ref{eqn-u*}) with respect to $y_n$ $[p]$ times and each derivative $D^m_{y'}D^l_{y_n}u^*$ with $ 3\leq l\leq [p]+2$ can be expressed as a polynomial 
involving nonnegative powers of $y_n$ and derivatives $D_{y'}^q D_{y_n}^su^*$ with $s< l$, thus $u^* \in C^{2+[p],\gamma} (\ov B_{\delta/2}^+)$ by an induction argument. It follows
that $u\in C^{2+[p],\beta}(\overline{\Omega})$ as desired. 
\end{myindentpar}
 \end{proof}

\end{document}